\documentclass[11pt]{article}
\usepackage[margin=0.75in]{geometry}

\usepackage{amssymb,amsmath,mathtools,amsthm,thmtools,amsfonts}

\usepackage{comment}
\usepackage{thm-restate}
\usepackage{ dsfont }

\usepackage{url} 
\usepackage{hyperref}
\hypersetup{
colorlinks,
linkcolor={red}, 
citecolor={black},
urlcolor={blue!60!black},
pdftitle={A Combinatorial Proof of Hilton’s Conjecture and Beyond},
pdfauthor={Thomas Lesgourgues, Luke Postle}}

\usepackage[utf8]{inputenc}
\usepackage{color}

\usepackage{enumitem}
\setlist[itemize]{topsep=.1in,itemsep=.1in}
\setlist[enumerate]{topsep=.1in,itemsep=.1in}
\setlist[enumerate,1]{label={(\roman*)}}
\setlist[enumerate,2]{label={(\alph*)}}
\setlist[enumerate,3]{label={(\arabic*)}}

\setlist[itemize]{nolistsep,noitemsep, topsep=0pt}

\newcommand{\eps}{\varepsilon}
\newcommand{\E}{\mathbb{E}}

\newcommand{\mc}[1]{\mathcal{#1}}%

\newcommand{\cF}{\ensuremath{\mathcal F}}

\newcommand{\cH}{\ensuremath{\mathcal H}}

\newcommand{\cQ}{\ensuremath{\mathcal Q}}

\newcommand{\cS}{\ensuremath{\mathcal S}}

\DeclareMathOperator{\on}{\mc{O}{\rm n}}
\DeclareMathOperator{\off}{\mc{O}{\rm ff}}
\DeclareMathOperator{\hleft}{\mc{L}}
\DeclareMathOperator{\hright}{\mc{R}}

\newcommand{\Expect}[1]{\ensuremath{\mathbb E\left[#1\right]}}
\newcommand{\Prob}[1]{\ensuremath{\mathbb P\left[#1\right]}}

\DeclarePairedDelimiter{\parens}{(}{)}
\DeclarePairedDelimiter{\set}{\{}{\}}
\DeclarePairedDelimiter{\brackets}{[}{]}

\newcommand{\TS}{\text{Tur{\'a}nSpace}}

\usepackage[noabbrev,capitalise]{cleveref}

\theoremstyle{plain}
\newtheorem{thm}{Theorem}[section]
\newtheorem{lem}[thm]{Lemma}
\newtheorem{proposition}[thm]{Proposition}

\newtheorem{cor}[thm]{Corollary}

\newtheorem{conj}[thm]{Conjecture}

\newenvironment{lateproof}[1]
 {%
  \renewcommand{\theclaim}{\ref{#1}.\arabic{claim}}%
  \begin{proof}[Proof of~\cref{#1}]%
 }
 {\end{proof}}

\usepackage{etoolbox}
 \AtEndEnvironment{proof}{\setcounter{claim}{0}}

\theoremstyle{plain} 
\newcommand{\thistheoremname}{}
\newtheorem{genericthm}{\thistheoremname}

\theoremstyle{definition}
\newtheorem{definition}[thm]{Definition}

\newtheorem{remark}[thm]{Remark}

\crefname{equation}{}{}
\crefname{lem}{Lemma}{Lemmas}
\crefname{claim}{Claim}{Claims}
\crefname{thm}{Theorem}{Theorems}
\crefname{enumi}{}{}

\usepackage{thm-restate}

\usepackage[dvipsnames, table]{xcolor}

\newcommand{\Erdos}{Erd\H{o}s }
\newcommand{\Kuhn}{K\"{u}hn}

\newcommand{\Turan}{Tur{\'a}n }

\title{On the Hypergraph Nash-Williams' Conjecture}

\author{
Cicely (Cece) Henderson
\thanks{Combinatorics and Optimization Department,
University of Waterloo, Waterloo, Ontario N2L 3G1, Canada {\tt c3hender@uwaterloo.ca}.}
\and
Luke Postle
\thanks{Combinatorics and Optimization Department,
University of Waterloo, Waterloo, Ontario N2L 3G1, Canada {\tt lpostle@uwaterloo.ca}. Partially supported by NSERC
under Discovery Grant No. 2019-04304.}}
\date{\today}

\begin{document}

\maketitle

\begin{abstract} 
In 2014, as part of his famous resolution of the Existence Conjecture for Combinatorial Designs, Keevash proved the existence of $(n,q,r)$-Steiner systems (equivalently $K_q^r$-decompositions of $K_n^r$) for all large enough $n$ satisfying the necessary divisibility conditions. In 2021, Glock, K\"uhn, and Osthus proposed a generalization of this result. Namely they conjectured a hypergraph version of Nash-Williams' Conjecture positing that if a $K_q^r$-divisible $r$-graph $G$ on $n$ vertices has minimum $(r-1)$-degree (denoted $\delta(G)$ hereafter) at least $\left(1-\Theta_r\left(\frac{1}{q^{r-1}}\right)\right) \cdot n$, then $G$ admits a $K_q^r$-decomposition.

The best known progress on this conjecture dates to the second proof of the Existence Conjecture by Glock, K\"uhn, Lo, and Osthus wherein they showed that $\delta(G)\ge \left(1-\frac{c}{q^{2r}}\right)\cdot n$ suffices for large enough $n$, where $c$ is a constant depending on $r$ but not $q$. More is known for the fractional relaxation. In 2017, Barber, K\"uhn, Lo, Montgomery and Osthus proved that $\delta(G)\ge \left(1-\frac{c}{q^{2r-1}}\right) \cdot n$ guarantees a fractional $K_q^r$-decomposition. Quite recently, Delcourt, Lesgourgues, and the second author improved this fractional bound to $\left(1-\frac{c}{q^{r-1 + o(1)}}\right)\cdot n$.

We prove that for every integer $r\ge 2$, there exists a real $c>0$ such that if a $K_q^r$-divisible $r$-graph $G$ satisfies $\delta(G)\ge \max\left\{ \delta_{K_q^r}^* + \eps,~~1 -\frac{c}{\binom{q}{r-1}} \right\} \cdot n$, then $G$ admits a $K_q^r$-decomposition for all large enough $n$, where $\delta_{K_q^r}^*$ denotes the fractional $K_q^r$-decomposition threshold. Combined with the fractional result above, this proves that $\left(1-\frac{c}{q^{r-1 + o(1)}}\right)\cdot n$ suffices for the Hypergraph Nash-Williams' Conjecture, approximately confirming the correct order of $q$. Our proof uses the newly developed method of refined absorption; we also develop a non-uniform Tur\'an theory to prove the existence of many embeddings of absorbers which may be of independent interest. 

\end{abstract}

\section{Introduction}
The study of combinatorial designs includes some of the oldest questions at the heart of combinatorics. In one of the first, most fundamental results in the area, Kirkman \cite{KirkmanSTS} proved in 1847 that there exists a set of triples of an $n$-set $X$ such that every pair in $X$ is in exactly one triple if and only if $n$ satisfies the necessary conditions of $3~|~\binom{n}{2}$ and $2~|~(n-1)$ (equivalently, $n \equiv 1, 3 \mod 6$);
such a set is called a \emph{Steiner Triple System}. From the graph-theoretic perspective, this is equivalent to a decomposition of the edges of the complete graph $K_n$ into edge-disjoint triangles. More generally, given integers $n > q > r \geq 1$, an $(n,q,r)$\textit{-Steiner System} is a set $S$ of $q$-subsets of an $n$-set $X$ such that each $r$-subset of $X$ is contained in exactly one element of $S$. From the hypergraph-theoretic perspective, an $(n,q,r)$-Steiner system is equivalent to a decomposition of the edges of $K_n^r$ (the complete $r$-uniform hypergraph on $n$ vertices) into edge-disjoint copies of $K_q^r$ (referred to as cliques). As before, for fixed values of $q$ and $r$, there are natural divisibility conditions on $n$ for the existence of an $(n, q, r)$-Steiner System: for each $0 \leq i \leq r-1,$ one requires that $\binom{q-i}{r-i}~|~\binom{n-i}{r-i}$.

Since the mid-1800s, the central question in design theory was the infamous Existence of Combinatorial Designs Conjecture. An important special case of the conjecture posited that a sufficiently large, divisible complete hypergraph $K_n^r$ always admits a $K_q^r$-decomposition as follows.

\begin{conj}[Existence Conjecture for Steiner Systems]
    For all integers $q > r \ge 1$, if $n$ is sufficiently large and $\binom{q-i}{r-i}~|~\binom{n-i}{r-i}$ for each $0 \leq i \leq r-1,$ then an $(n, q, r)$-Steiner system exists. 
\end{conj}

More generally, the Existence Conjecture posited the existence of an \emph{$(n,q,r,\lambda)$-design} where each $r$-subset is in $\lambda$ elements of $S$ for some $\lambda\ge 1$; that said, the $\lambda=1$ case was viewed as the hardest case by far. Nearly two centuries after it entered the literature as folklore, it was resolved in the affirmative by Keevash~\cite{keevashEC}.
Throughout the 19th and 20th centuries, several individuals made significant progress on special cases of the conjecture. In particular, Wilson~\cite{WilsonCasesI, WilsonCasesII, WilsonCasesIII} revolutionized design theory in the 1970s when he completed the graph case by proving existence for all values of $q$ when $r = 2$ in a series of papers. 
In 1985 R{\"o}dl \cite{Nibble} introduced the celebrated nibble method, featured in this paper, in order to settle the approximate version of the problem, also known as the Erd{\H o}s-Hanani conjecture \cite{ApproxExistenceConj}. Until 2014 though, no $(n,q,r)$-Steiner systems were known to exist for any $r > 5$.

Thus it was only in the last decade that the Existence Conjecture was solved by Keevash~\cite{keevashEC}. Later, several proofs emerged, each using different techniques. Keevash's original 2014 argument uses \textit{randomized algebraic constructions}. In 2016, Glock, K{\"u}hn, Lo, and Osthus~\cite{GKLO16} gave a purely combinatorial proof of the Existence Conjecture via \textit{iterative absorption}. Both approaches have different benefits and each led to subsequent progress on various problems. In 2024, Delcourt and the second author~\cite{DPI} developed the methodology of \textit{refined
absorption}, giving a one-step, combinatorial proof of the Existence Conjecture. Later in 2024, Keevash~\cite{KeevashShortEC} 
provided a more concise proof that also uses the new refined absorption framework.

\subsection{History of Nash-Williams' Conjecture}\label{Sec:NW}

Given the resolution of the Existence Conjecture, a natural question is when a partial Steiner system (i.e., a set of edge-disjoint cliques) can be extended to a Steiner system. More generally, one may wonder when \emph{nearly} complete graphs admit a decomposition, provided they are divisible. Since the Existence Conjecture was only recently resolved in full while the $r = 2$ case has been settled since Wilson's~\cite{WilsonCasesI, WilsonCasesII, WilsonCasesIII} work in the 1970s, graph decompositions of non-complete graphs are the most well-studied.
A graph $G$ is $K_q$\textit{-divisible} if $\binom{q}{2}\;|\;e(G)$ and $(q-1)\;|\;\deg(v)$ for every vertex $v \in V(G)$. However, if $G$ is not complete, divisibility is not enough to guarantee decomposability. For example, $C_6$ is $K_3$-divisible but not $K_3$-decomposable.

When considering designs in a non-complete graph $G$, the higher the minimum degree is as a proportion of the number of vertices, the ``closer'' $G$ is to being complete and hence addressed by the $r=2$ case of the Existence Conjecture. Therefore, a natural question is whether a high minimum degree in a divisible graph is enough to guarantee the existence of a decomposition. Let $\delta_{K_q}$ denote the \textit{$K_q$-decomposition threshold}, defined as $\lim\sup_{n\rightarrow \infty} \delta_{K_q}(n)/n$ where $\delta_{K_q}(n)$ is the smallest integer $d$ such that every $K_q$-divisible graph $G$ on $n$ vertices with $\delta(G)\ge d$ has a $K_q$-decomposition. There is an analogous definition for the $K_q^r$-decomposition threshold for any choice of $q > r \ge 1$. Nash-Williams' Conjecture~\cite{NWConj} from 1970 is a central open question in extremal design theory that addresses the case of triangle decompositions.

\begin{conj}[Nash-Williams' Conjecture \cite{NWConj}]
    If $G$ is a sufficiently large $K_3$-divisible graph and $\delta(G) \geq \frac{3}{4}v(G)$, then $G$ admits a triangle decomposition.
\end{conj}

In an addendum to Nash-Williams' paper, Graham provided a construction showing that this minimum degree threshold would be tight, and and hence $\delta_{K_3} \ge 3/4.$ 

Nash-Williams' Conjecture is at the heart of extremal combinatorics as it is firmly situated inside a larger framework that encompasses many of the principal theorems in the field. In particular, the threshold for the number of edges needed to guarantee even one triangle is a famous classical result widely attributed to Mantel~\cite{mantel} from 1907 who showed that every graph on $n$ vertices with at least $\lfloor \frac{n^2}{4} \rfloor +1$ edges contains a triangle (which is tight). In a seminal result from 1941, Tur{\'a}n~\cite{turan1941extremal} generalized this by determining the exact number of edges needed to guarantee the existence of a $K_q$ in a graph - roughly around $\big(1-\frac{1}{q-1}\big)n$. This became known as Tur\'an's theorem, the numbers as Tur\'an numbers, and launched an entire subfield of extremal combinatorics known as Tur\'an theory.

Furthermore, there is a rich history on the study of density thresholds for the existence of a collection of vertex-disjoint triangles spanning all vertices (called a \emph{triangle factor}). In this case, as with decomposition thresholds, it is natural to impose a minimum degree condition in order to prevent the existence of isolated vertices. In 1963, Corr{\'a}di and Hajnal~\cite{CH63} famously resolved the question of the minimum degree threshold for a triangle factor by showing that any graph $G$ on $n$ vertices with $3 ~|~n$ (a necessary divisibility condition) and $\delta(G) \ge \frac{2}{3}\cdot n$ contains a triangle factor. As for the minimum degree necessary to force the existence of a \emph{$K_q$-factor} (a collection of vertex-disjoint copies of $K_q$ spanning all vertices), the Hajnal-Szemer\'edi Theorem~\cite{HS70} from 1970 implies that any graph $G$ on $n$ vertices with $q ~|~n$ and and $\delta(G) \ge \frac{q-1}{q}\cdot n$ contains a $K_q$-factor.

Nash-Williams' Conjecture extends this deeply rooted lineage. Most directly, the existence of a $K_3$-subgraph from Mantel's Theorem is a necessary condition for both the existence of a triangle factor as in the Corr{\'a}di- Hajnal Theorem and for the existence of a triangle decomposition as in Nash-Williams' Conjecture. Yet there is also a more fundamental perspective that unites these results: for each of the above theorems, there exists an $i \in \{0, 1, 2\}$ such that the result provides the minimum density threshold to guarantee there exist a collection of $K_q$-subgraphs of a divisible graph $G$ such that every $i$-set of $G$ is in exactly one element in the set (where we view $G$ as a simplicial complex). Since a copy of $K_q$ contains the empty set, Tur\'an's theorem answers the $i = 0$ case, while the Hajnal-Szemer\'edi Theorem answers the $i = 1$ case. Nash-Williams' Conjecture thus extends this natural ``simplicial complex'' question to the case when $q = 3$ and $i = 2$. This notion also generalizes to collections of $K_q^r$-subgraphs in $r$-uniform hypergraphs for all $i \in [r]_0$ which we discuss in more detail later.

Nash-Williams' Conjecture is also a generalization of yet another fundamental graph theory result: in 1952, Dirac \cite{DiracThm} showed that if $G$ is a graph with at least three vertices and $\delta(G) \ge \frac{1}{2}v(G)$, then $G$ is Hamiltonian. From a different perspective, this implies that if $G$ satisfies the necessary divisibility conditions for the existence of a perfect matching (i.e., $v(G)$ is even) and $\delta(G) \ge \frac{1}{2}v(G)$, then $G$ has a perfect matching. Since a perfect matching is equivalent to a $(n, 2, 1)$-Steiner System, we can reframe Dirac's Theorem as a minimum degree threshold for which a necessary condition for the existence of a $(n, 2, 1)$-Steiner System becomes sufficient. Nash-Williams' Conjecture generalizes this notion to a $(n, 3, 2)$-Steiner System. 

The state-of-the-art for Nash-Williams' Conjecture is that $\delta_{K_3} \le \frac{7+\sqrt{21}}{14} \approx 0.82733$, which follows from the combination of two results. First, Barber, K{\"u}hn, Lo, and Osthus~\cite{BKLOr=3} tied the $K_3$-decomposition threshold to the fractional $K_3$-decomposition threshold as follows. A \textit{fractional $F$-decomposition} of $G$ is a collection of $F$-subgraphs of $G$ each with a weight in $[0, 1]$ such that the total weight on each edge $e \in E(G)$ over all copies of $F$ containing $e$ is exactly 1.
The \emph{fractional $K_3$-decomposition threshold} $\delta^*_{K_3}$ is defined as the infimum of all real numbers $c$ such that every graph $G$ with minimum degree at least $c\cdot v(G)$ admits a fractional $K_3$-decomposition.
(Again, there is an analogous definition for the fractional $K_q^r$-decomposition threshold $\delta^*_{K_q^r}$.) Fractional $F$-decomposability is a necessary condition for a graph to be $F$-decomposable, and Graham's construction in the addendum to~\cite{NWConj} in fact yields that $\delta_{K_3}^* \ge \frac{3}{4}.$ In~\cite{BKLOr=3}, Barber, K{\"u}hn, Lo, and Osthus used the method of iterative absorption to show that for each real $\varepsilon>0$, every sufficiently large $K_3$-divisible graph~$G$ on~$n$ vertices with $\delta(G)\ge (\max\{\delta^*_{K_3},\frac{3}{4}\}+\varepsilon)n$ admits a $K_3$-decomposition. Delcourt and the second author~\cite{DELCOURT2021382} later improved the upper bound on $\delta_{K_3}^*$ to $\frac{7+\sqrt{21}}{14}$ in 2020, thus obtaining the best known bound for $\delta_{K_3}$.

\subsection{Graph Nash-Williams' Conjecture}

A folklore generalization of Nash-Williams' Conjecture to $K_q$-decompositions (dating at least to Gustavsson's 1991 thesis~\cite{gustavsson1991decompositions}) posited that if $G$ is a sufficiently large $K_q$-divisible graph with $\delta(G) \ge \frac{q}{q+1}$, then $G$ admits a $K_q$-decomposition. A construction from Gustavsson~\cite{gustavsson1991decompositions} showed that this minimum degree threshold would have been tight. In 2019, Glock, \Kuhn, Lo, Montgomery, and Osthus~\cite{GKLMO19} tied the $K_q$-decomposition threshold to the fractional $K_q$-decomposition threshold by showing that for all $q \ge 3$, $\delta(G) \ge \left(\max \left\{\delta_{K_q}^*, \frac{q}{q+1} \right\} +\eps\right)n$ suffices to guarantee $G$ admits a $K_q$-decomposition. In the same paper, they also showed that the fractional $K_q$-decomposition threshold is at most $1-\frac{1}{10^4q^{3/2}}$, before Montgomery \cite{MontgomeryCliqueDecomps} improved this bound to $1- \frac{1}{100q}$ in 2019. For each $q \ge 4$, Delcourt, Lesgourgues, and the authors \cite{DHLP25Counter} disproved the folklore conjecture in 2025 by providing a construction of an infinite family of $K_q$-divisible graphs $G$ with minimum degree at least $\left(1-\frac{1}{c\cdot(q+1)}\right)\cdot v(G)$ for a constant $c >1$ that admit neither a $K_q$-decomposition nor even a fractional $K_q$-decomposition. In fact, Delcourt, Lesgourgues, and the authors even proved that the conjecture is off by a multiplicative factor for large enough $q$ by constructing infinitely many $K_q$-divisible graphs $G$ with minimum degree at least $\bigg(1-\frac{1}{\left(\frac{1+\sqrt{2}}{2}-\varepsilon\right)\cdot (q+1)}\bigg)v(G)$ with no $K_q$-decomposition.

\subsection{Hypergraph Nash-Williams' Conjecture}

In 2021, Glock, K{\"u}hn, and Osthus \cite{GKO20Survey} proposed a hypergraph generalization of Nash-Williams' Conjecture. First, we require more definitions and notation.
For brevity, an $r$-uniform hypergraph is an \textit{$r$-graph}. An $r$-graph $G$ is $K_q^r$\textit{-divisible} if for every $i$ with $0 \le i \le {r-1}$ and $T \subseteq \binom{V(G)}{i}$, we have that $\binom{q-i}{r-i} ~|~ e(G(T)).$ For an $r$-graph $G$, we let $\Delta(G):=\Delta_{r-1}(G)$, the maximum over all $(r-1)$-sets $S$ of the number of edges of $G$ containing $S$. Similarly, $\delta(G):=\delta_{r-1}(G)$ denotes the minimum $(r-1)$-degree of $G.$

By modifying a construction from \cite{OnIndepSets}, Glock, K{\"u}hn, Lo, and Osthus showed in \cite{GKLO16} that there exist infinitely many $K_q^r$-free $r$-graphs $G$ with $\delta(G) \geq \left(1 - c\cdot \frac{\log q}{q^{r-1}}\right)\cdot v(G)$ for some constant $c>0$ that depends on $r$ but not $q$, where this bound is derived from results in hypergraph Tur{\' a}n theory. Since the existence of a $K_q^r$-subgraph is necessary for the existence of a $K_q^r$-decomposition, hypergraph \Turan theory offers valuable insights into the problem of determining $\delta_{K_q^r}$. Yet, despite being a central question in extremal combinatorics, the problem of determining hypergraph \Turan numbers is notoriously difficult. For general $q$ and $r$, the best-known bounds on the \Turan number of $K_q^r$ are due to de Caen~\cite{dC83}, who proved an upper bound of $1 - \frac{1}{\binom{q-1}{r-1}}$, and Sidorenko~\cite{S81}, who proved a lower bound of $1 - \left(\frac{r-1}{q-1}\right)^{r-1}$. Note that these results together imply that the \Turan number for $K_q^r$ is $1 - \Theta_r\left(\frac{1}{q^{r-1}} \right)$, thus settling the order for fixed $r$ and large $q$. 

Since the \Turan number concerns only the number of edges, it is natural in hypergraphs to study alternate density conditions. The \textit{codegree \Turan density} is the minimum $\gamma >0$ such that any $r$-graph on $n$ vertices with $\delta(G) \ge (\gamma + o(1))\cdot n$ contains $F$ as a subgraph. While Glock, \Kuhn, Lo, and Osthus' construction provides a lower bound on the codegree \Turan density of $K_q^r$, in 2018 Lo and Zhao~\cite{LZ18} derived an upper bound of the same order, thus determining the codegree \Turan density of $K_q^r$ to be $1- \Theta_r\left(\frac{\log q}{q^{r-1}}\right)$ (which is striking in that the $\log q$ factor is needed for the codegree \Turan density when it is not for the $K_q^r$ \Turan number). As for the hypergraph version of Hajnal-Szemer\'edi, a result of Lu and Szekely~\cite{LS07} from 2007 implies that there exists for some constant $c > 0$ that depends on $r$ but not $q$ such that if $\delta(G) \ge \left(1 - \frac{c}{q^{r-1}}\right)\cdot n$ , then $G$ contains a $K_q^r$-factor. The constant on this bound was improved in 2021 by Pavez-Sign{\'e}, Sanhueza-Matamala, and Stein~\cite{PSS21} in their work on a hypergraph version of the Pósa-Seymour conjecture; we note that an alternate proof of their results for was given by Lang and Sanhueza-Matamala~\cite{LS24} in 2024. That said, our knowledge the best known lower bound for the hypergraph version of Hajnal-Szemer\'edi still comes from the codegree Tur\'an bound. In light of this history, the following conjecture of Glock, K{\"u}hn, and Osthus~\cite{GKO20Survey} is natural. 

\begin{conj}[Conjecture 4.4 in \cite{GKO20Survey}]\label{Conj:HG}
   Let $q > r \ge 3$ be integers. If $G$ is a sufficiently large $K_q^r$-divisible $r$-graph such that $\delta(G) \geq \left(1 - \Theta_{r}\left(\frac{1}{q^{r-1}}\right)\right)\cdot v(G)$ then $G$ admits a $K_q^r$-decomposition.
\end{conj}
Note that Glock, K{\"u}hn, and Osthus only conjectured the order of $\delta_{K_q^r}$, which is reasonable considering the notorious difficulty of determining constants in hypergraph \Turan problems. Also note that they specifically conjectured the above without the additional $\log q$ factor, suggesting that a better lower bound construction for~\cref{Conj:HG} should exist than what is achievable for the codegree Tur\'an density of $K_q^r$. We find this reasonable in light of our main results and our understanding of rooted Tur\'an theory; on the other hand we have reason to believe that some form of logarithmic factor may be achievable in the upper bound for the hypergraph Hajnal-Szemer\'edi problem (and perhaps even for some of the less-studied intermediate problems of the minimum codegree needed to guarantee the existence of a set of copies of $K_q^r$ covering each $i$-set exactly once for $2\le i \le r-1$). 

As for what is known about the upper bound in~\cref{Conj:HG}, there is a large gap between the conjectured bound for the $K_q^r$-decomposition threshold and the best-known upper bound given by Glock, K{\"u}hn, Lo, and Osthus in \cite{GKLO16} in the second proof of the Existence Conjecture as follows.

\begin{thm}[Glock, K{\"u}hn, Lo, and Osthus \cite{GKLO16}]\label{thm:GKLOLowerBound}
    Let $q > r \ge 3$ be integers. All sufficiently large $K_q^r$-divisible $r$-graphs with $\delta(G) \geq\left(1 - \frac{r!}{3\cdot14^rq^{2r}}\right)\cdot v(G)$ admit a $K_q^r$-decomposition.
\end{thm}

As for fractional decompositions, Barber, K{\"u}hn, Lo, Montgomery, and Osthus~\cite{BKLMOFracHGNW} showed in 2017 that $\delta^*_{K_q} \le 1 - \frac{c}{q^{2r-1}}$ for all $q > r \ge 1$ for some constant $c$ depending on $r$ but not $q$. In 2025, Delcourt, Lesgourgues, and the second author \cite{DLP25} dramatically improved this bound as follows. 

\begin{thm}[Delcourt, Lesgourgues, and Postle \cite{DLP25}]\label{thm:DLPFractionalBound}
For every integer $r \geq 3$ and real $\varepsilon\in (0,1]$, there exists a constant $c > 0$ such that $\delta^*_{K_q^r}\leq 1-\frac{c}{q^{r-1+\varepsilon}}$ for every integer $q > r$. 
\end{thm}

The main result of this paper is to tie the decomposition threshold $\delta_{K_q^r}$ to its fractional relaxation.
\begin{thm}\label{thm:main}
    For every integer $r \ge 3$, there exists a real $c >0$ such that the following holds. For every integer $q > r$ and real $\eps > 0$,
    every sufficiently large $K_q^r$-divisible $r$-graph $G$ on $n$ vertices with minimum degree 
    \[ \delta(G) \ge \max\left\{ \delta_{K_q^r}^* + \eps,~~1 -\frac{c}{\binom{q}{r-1}} \right\} \cdot n,\]
    admits a $K_q^r$-decomposition.
\end{thm}
In particular, we obtain the constant factor $c = \frac{1}{2^{2r+2} \cdot 3^r \cdot r^{r+1}}$. We note that the $\eps$ in the above statement ensures there exists a fractional $K_q^r$-decomposition after removing a low maximum degree subgraph of $G$, and hence is only necessary for the regularity boosting step of our proof. Additionally, the ``maximum'' is necessary because we have neither a sufficient lower nor upper bound on $\delta_{K_q^r}^*$ to definitively compare between the two terms. 

When combined with~\cref{thm:DLPFractionalBound}, Theorem~\ref{thm:main} yields the following corollary, which dramatically decreases the gap between what is known and~\cref{Conj:HG} by showing that $\delta_{K_q^r}$ is almost of the same order as was conjectured by Glock, K{\"u}hn, Lo, and Osthus.

\begin{cor}\label{cor:main}
    For every integer $r \ge 1$, there exists a real $c >0$ such that the following holds. For every integer $q > r$ and real $\eps > 0$,
    every sufficiently large $K_q^r$-divisible $r$-graph $G$ on $n$ vertices with minimum degree 
    \[ \delta(G) \ge \left(1 -\frac{c}{q^{r-1+\eps}}\right) \cdot n,\]
    admits a $K_q^r$-decomposition.
\end{cor}

\subsection{On the Novelty of Our Proof}\label{subsec:novelty}

One novelty of the arguments developed in this paper is that we establish a new non-uniform \Turan theory, thus overcoming the challenges related to embedding hypergraph absorbers into a high minimum degree host hypergraph. This culminates in a streamlined proof using the new technique of refined absorption.

The main obstacle to proving~\cref{thm:main} using existing strategies is due to the properties of hypergraph absorbers. In the proof of~\cref{thm:GKLOLowerBound}, the minimum degree of the host hypergraph $G$ is high relative to the rooted degeneracy of each hypergraph absorber, so the absorbers can be embedded one vertex at a time. Since the sparsest known constructions for hypergraph absorbers, as found in~\cite{DKP25}, have rooted degeneracy $\Theta_r(q^{2r-2})$ (and there are reasons to believe the theoretical limit for the rooted degeneracy of $K_q^r$-absorbers is at least on the order of $q^{2r-3}$), this presents a bottleneck for improving the minimum degree threshold in~\cref{thm:GKLOLowerBound}. Moreover, the proof of~\cref{thm:GKLOLowerBound} uses the method of iterative absorption, which relies both on many cover down steps but also an inductive argument on the uniformity. Thus, even if the rooted degeneracy of hypergraph absorbers could be improved, modifying the proof of~\cref{thm:GKLOLowerBound} would require redoing their entire argument while still resolving the obstacles we tackle here.

Thus we prove~\cref{thm:main} using the technique of refined absorption. Introduced by Delcourt and the second author~\cite{DPI} in their 2024 one-step combinatorial proof of the Existence Conjecture, this method allows us to take the existence of powerful absorbing structures called \emph{omni-absorbers} as a black box. Hence, we bypass the need to modify every step of the existence proof, and so obtain a shorter, more streamlined argument than if we used iterative absorption. In fact, refined absorption has been successfully modified for decomposition proofs in multiple settings, including the high-girth Nash-Williams' graph setting by Delcourt, Lesgourgues, and the authors in \cite{EMeetsNW}. 

Yet, the previous modifications to refined absorption are not enough to prove~\cref{thm:main}: in order to use the existence of omni-absorbers in $K_n^r$ as a black box as in~\cite{EMeetsNW}, we would need to be able to embed private absorbers to ensure the final construction uses only edges present in $G$. While rooted degeneracy suffices to embed graph absorbers in~\cite{EMeetsNW}, it remains a hurdle in our hypergraph setting. Our solution is motivated by the the natural partiteness of the absorber construction from~\cite{DKP25} and the Erd{\H o}s-Stone–Simonovits theorem (first explicitly proven and linked to Tur\'an numbers by Erd\H{o}s and Simonovits~\cite{ES66} in 1966, but following easily from the earlier work of Erd\H{o}s and Stone \cite{ES46} from 1946), which bounds the maximum number of edges in a graph that does not contain a blow up of a complete partite graph.
Combined with the classical notion of ``supersaturation'' introduced by \Erdos and Simonovits \cite{ES82}, we show that if our host hypergraph $G$ has high enough edge density, it would guarantee the existence of many copies of an absorber. However, even this is not enough, as we also require the additional property that our absorbers be rooted in $G$ at various subsets of $G$ (specifically cliques whose edges not in $G$ have been replaced by `fake-edge' gadgets in $G$). Since each edge in a hypergraph absorber may contain a different number of root vertices, we generalize the ideas above to a non-uniform setting to find many copies of an absorber modulo its root vertices. Thus, we overcome the obstacle of absorber embedding by establishing a non-uniform \Turan theory which includes non-uniform generalizations of classical results on supersaturation, the \Turan number of blow-up graphs, and the \Turan number of $K_q^r$, and hence, may be of independent interest. 

\subsection{Notation and Definitions}\label{subsec:notation}

A hypergraph $G$  consists of a set $V(G)$ whose elements are called \textit{vertices} and a set $E(G)$ of subsets of $V(G)$ called \textit{edges}; for brevity, we also sometimes write $G$ for its set of edges $E(G)$ and for a set $U \subseteq E(G)$, we sometimes write $V(U)$ for the set of vertices $\bigcup_{e\in U}e$. Similarly, we write
$v(G)$ for the number of vertices of $G$ and $e(G)$ for the number of edges of $G$.
A \textit{multi-hypergraph} consists of a set of vertices $V(G)$ and a multi-set $E(G)$ of subsets of $V$. 
For an integer $r \geq 1$, a (multi-)hypergraph $G$ is said to be \textit{$r$-bounded} if every edge of $G$ has size at most $r$ and \textit{$r$-uniform} if every edge has size exactly $r$. If $G$ is an $r$-uniform hypergraph, we say $G$ is an \textit{$r$-graph.}
For a set $S \subseteq V(H)$, we let $G(S) \coloneqq \{e \in E(G) : e \supseteq S\}$.
We denote the \textit{degree} of a set $S \subseteq V(G)$ as $d_G(S) \coloneqq |G(S)|$, and for $i \in \mathbb N$, we let $\Delta_i(G) \coloneqq \max_{S \in \binom{V(H)}{i}}d_G(S)$.
For an $r$-graph $G$, we let $\Delta(G) \coloneqq \Delta_{r-1}(G)$.
We also denote the \textit{degree} of a vertex $v \in V(G)$ as $d_G(v) \coloneqq d_G(\{v\})$. 
We let $G[S]$ denote the hypergraph of $G$ \textit{induced by} $S$ where $V(G[S]) \coloneqq S$ and $E(G[S]) \coloneqq \{e \in G : e \subseteq S\}$.

An $r$-uniform hypergraph $G$ is \textit{complete} if $E(G) = \binom{V(G)}{r}$, and a \textit{clique} in $G$ is a set $S \subseteq V(H)$ such that $G[S]$ is complete.  We call a clique $S$ of $G$ a \textit{$q$-clique} if $|S| = q$.
We let $K^r_q$ denote the complete $r$-uniform hypergraph with vertex set $[q]$.  

For hypergraphs $G$ and $H,$ we let $H \simeq G$ denote that $H$ is isomorphic to $G$. Likewise, $H \subseteq G$ denotes that $H$ is a subgraph of $G$ and $\binom{G}{H}$ denotes the set of $H$-subgraphs of $G$.

For a positive integer $r$, we use $[r]$ to denote the set $\{1, 2, \ldots, r\}$, and we let $[r]_0 = [r]\cup \{0\}$. For a set $S$ and a positive integer $r$, we use $\binom{S}{r}$ to denote the set of subsets of size $r$ in $S$.

\section{Proof Overview}

The new proof of the Existence Conjecture by Delcourt and the second author in~\cite{DPI} introduced the novel method of refined absorption. We now recall the relevant definitions. 

\begin{definition}[Absorber] Let $L$ be a $K^r_q$-divisible $r$-graph. An $r$-graph $A$ is a \emph{$K^r_q$-absorber} for $L$ if $V(L)$ is independent in $A$ and both $A$ and $A\cup L$ admit $K^r_q$-decompositions.    
\end{definition}

\begin{definition}[Omni-Absorber]\label{def:OmniAbsorbers}
An $r$-graph $A$ is a \emph{$K^r_q$-omni-absorber} for an $r$-graph $X$ with \emph{decomposition family} $\cF_A$ and \emph{decomposition function} $\cQ_A$  if $X$ and $A$ are edge-disjoint and for every $K^r_q$-divisible subgraph $L$ of $X$, $\cQ_A(L)\subseteq \cF_A$ is a $K^r_q$-decomposition of $A\cup L$.
\end{definition}

The key to the method of refined absorption is a black-box theorem about the existence of efficient omni-absorbers that also have a `refinedness' property as follows. A $K^r_q$-omni-absorber $A$ is \emph{$C$-refined} if every edge of $X\cup A$ is in at most $C$ elements of $\cF_A$. Here is the key black-box theorem from~\cite{DPI} that is at the heart of our proof of~\cref{thm:main}.

\begin{thm}[Refined Efficient Omni-Absorber~\cite{DPI}]\label{thm:RefinedEfficientOmniAbsorber}
For all integers $q > r\ge 1$, there exist an integer $C\ge 1$ such that the following holds: If $X \subseteq K_n^r$ with $\Delta(X) \le \frac{n}{C}$ and we let $\Delta:= \max\left\{\Delta(X),~n^{1-\frac{1}{r}}\cdot \log n\right\}$, then there exists a $C$-refined $K_q^r$-omni-absorber $A\subseteq K_n^r$ for $X$ such that $\Delta(A)\le C \cdot \Delta$.     
\end{thm}

We now recall the outline of the new proof of the Existence Conjecture~\cite{DPI}, adapted to the Nash-Williams' setting in~\cite{EMeetsNW}. Note that though the results in~\cite{EMeetsNW} are restricted to the graph case, the outline of the argument generalizes to hypergraphs.

\vspace{0.05in}
Let $G \subseteq K_n^r$ be as in~\cref{thm:main}. 
\begin{enumerate}\itemsep.05in
    \item[(1)] `Reserve' a random subset $X$ of $E(G)$ uniformly with some small probability $p$.
    \item[(2)] Construct an omni-absorber $A$ of $X$ by taking an omni-absorber $A_0$ for $X$ as guaranteed by ~\cref{thm:RefinedEfficientOmniAbsorber}, and re-embed $A_0$ into $G$ using ``fake edges" and private absorbers.
    \item[(3)] ``Regularity boost'' $G\setminus (A\cup X)$ to find a very regular set of cliques.
    \item[(4)] Apply ``nibble with reserves'' theorem to find a $K^r_q$-packing of $G\setminus A$ covering $G\setminus (A\cup X)$ and then extend this to a $K^r_q$-decomposition of $G$ by definition of omni-absorbers.
\end{enumerate}

For (1), the authors of~\cite{DPI} proved that there is a choice of $X$ where every edge $e\in G\setminus X$ is in many copies of $K^r_q$ in $X\cup\{e\}$ and $\Delta(X)$ (and, hence, $\Delta(X\cup A)$ by~\cref{thm:RefinedEfficientOmniAbsorber}) is small, and this proof easily generalizes to the Nash-Williams's setting. Similarly, step (4) is a standard technique in absorption proofs whose use also directly generalizes from the complete setting of~\cite{DPI} to our Nash-Williams' setting. For (3), the authors in~\cite{DPI} use a special case of the Boosting Lemma of Glock, K\"{u}hn, Lo, and Osthus~\cite{GKLO16} which requires a minimum degree very close to $n$. This step is more complicated in the Nash-Williams' setting, but the solution presented in~\cite{EMeetsNW} generalizes from the graph Nash-Williams' setting to our hypergraph Nash-Williams' setting. 

The crux of our proof of~\cref{thm:main} then is in accomplishing (2). From~\cref{thm:RefinedEfficientOmniAbsorber}, there exists a refined $K_q^r$-omni-absorber $A_0$ for $X$ in $K_n^r$. Yet $A_0$ may use edges that are not present in $G$. The solution in~\cite{EMeetsNW} is to re-embed $A_0$ into $G \setminus X$, using the methodology of fake-edges and private absorbers developed in~\cite{DPI}. By the General Embedding Lemma also established in~\cite{EMeetsNW}, it suffices to show that there are many possibilities for where to embed each individual gadget into $G$. In the graph setting of~\cite{EMeetsNW}, this is accomplished by showing that both fake-edges and graph absorbers can be embedded one vertex at a time due to their small rooted degeneracy relative to the minimum degree of $G$. 
However, this strategy fails for hypergraphs because $G$ has missing degree on the order of $1/q^{r-1}$, while the sparsest known $K_q^r$-absorber construction, as presented in~\cite{DKP25}, yields a rooted degeneracy of $\Theta_r(q^{2r-2})$. 

We overcome this hurdle by exploiting the structure of hypergraph absorbers, rather than the rooted degeneracy. In particular, the absorbers in~\cite{DKP25} are built from layered gadgets, each of which is $q$-partite. The famous Erd{\H o}s-Stone–Simonovits theorem \cite{ES66, ES46} bounds the maximum number of edges in a graph that does not contain a $t$-blow up of a complete $q$-partite graph. Since the high minimum degree condition for $G$ guarantees it has high edge density, we are motivated to use \Turan theory to show that $G$ contains many copies of each absorber. Yet, this is insufficient, as we also require the absorbers be correctly rooted in $G$. Since each edge of our hypergraph absorber may have a different number of root vertices, we instead establish a novel non-uniform hypergraph \Turan theory. We use this theory to find many copies of the hypergraph obtained by ``projecting out'' the root vertices of our absorber in the hypergraph obtained from $G$ by ``projecting'' out the root vertices (we formalize this idea in~\cref{Subsec:PfOverviewAbsorbers}). Thus, instead of a \Turan point, we define the notion of a \emph{\Turan space}, which is roughly the set of points of edge-densities in each uniformity which guarantee the existence of a non-uniform hypergraph (see~\cref{Subsec:NonUnifTuran} for a formal definition). 

In this paper, we prove three key properties of \Turan spaces. First, we generalize \Erdos and Simonovits' \cite{ES82} ``supersaturation'' result, by showing that if a hypergraph has high enough edge density in every uniformity, it not only contains one copy of a non-uniform hypergraph $F$ but many. Next, we show that the \Turan density of an $r$-bounded hypergraph $F$ is the same as its blow-up. Finally, we show that there is a number in our \Turan space that is satisfied by our high minimum degree hypergraph $G$. 

In~\cref{subsec:PfOverviewNW}, we include the statements and intuition for how we resolve steps (1), (3), and (4) of the proof outline. In~\cref{Subsec:PfOverviewAbsorbers}, we include the statements that allow us to reduce step (2) of the proof outline to an embedding problem. Then, in~\cref{Subsec:NonUnifTuran} we develop the Non-Uniform \Turan Theory to solve this embedding problem. 

\subsection{A Refined Absorption Proof of Hypergraph Nash-Williams' - Steps 1, 3, and 4}\label{subsec:PfOverviewNW}

As described above, steps (1), (3), and (4) of our proof outline are standard applications of previously developed techniques. Here we include the relevant statements for these steps. 

For (1), the subsequent statement follows from Chernoff bounds. We include its proof in~\cref{sec:Finish}.

\begin{lem}[Hypergraph High Minimum Degree Reserves Lemma]\label{lem:highmindegreserves}
For every $q > r \geq 1$ and real $\eps \in (0, 1)$, there exists a real number $\gamma \in (0, 1)$ and positive integer $n_0$ such that for all integers $n \geq n_0$ the following holds: 
\vskip.1in

\noindent Let $G \subseteq K_n^r$ such that $\delta(G) \geq \left(1 - \frac{1}{\binom{q-1}{r-1}}+\varepsilon\right)n$. If $p\in [n^{-\gamma}, 1)$, then there exists a spanning subgraph $X \subseteq G$ with $\Delta(X) \leq 2pn$ such that every $e \in G\setminus X$ is in at least $\gamma \cdot p^{\binom{q}{r}-1}\binom{n}{q-r}$ $K_q^r$-cliques of $X \cup \{e\}$.
\end{lem}

For (3), we follow in the footsteps of~\cite{EMeetsNW} and prove the following statement.  

\begin{thm}[Hypergraph Nash-Williams' Boosting Theorem - Polynomially Small Irregularity Version]\label{thm:NWRegBoost}
For every integers $q> r \ge 2$ and real $\varepsilon \in (0,1)$, there exists $c\in(0,1)$ such that the following holds for all $n$ large enough. If $J\subseteq K^r_n$ with minimum degree at least $(\max\{\delta_{K^r_q}^*, 1 - \frac{1}{\binom{q-1}{r-1}} \}+\varepsilon)n$, then there exists a family $\mc{H}$ of copies of $K^r_q$ in $J$ such that every $e\in J$ is in $\parens*{c\pm n^{-(q-r)/3}}\cdot \binom{n-r}{q-r}$ copies of $K^r_q$ in $\mc{H}$. 
\end{thm} 

As in~\cite{EMeetsNW}, \cref{thm:NWRegBoost} is a corollary of the following lemma (via sampling and Chernoff bounds). We say a fractional $F$-decomposition of an $r$-graph $G$ on $n$ vertices is \emph{$C$-low-weight} if every copy of $F$ in $G$ has weight at most $\frac{C}{\binom{n-r}{v(F)-r}}$. 

\begin{thm}\label{lem:LowWeightFracDecomposition}
    For every $r$-graph $F$ and real $\varepsilon\in (0,1)$, there exists an integer $C\geq 1$ such that the following holds for every integer $n$ large enough. If $G$ is a $n$-vertex $r$-graph with $\delta(G)\geq (\delta^*_F+\varepsilon)n$, then there exists a $C$-low-weight fractional $F$-decomposition of $G$.
\end{thm}

Following the strategy from~\cite{EMeetsNW}, in order to prove~\cref{lem:LowWeightFracDecomposition}, we first show that we can find an $F$-packing in a hypergraph $G$ where each edge obtains a predetermined ``target'' weight in the packing. Then, using an Inheritance Lemma from Lang~\cite{lang2023tiling}, we show that for an appropriately chosen integer $s$, every edge $e \in G$ is in many $s$-sets $S$ that roughly inherit the minimum degree of $G$, and hence $G[S]$ will admit a fractional $K_q^r$-decomposition. By ensuring the weight of $e$ is low in each of these decompositions and scaling appropriately, we obtain the desired low-weight fractional decomposition. The proofs of~\cref{lem:highmindegreserves} and~\cref{lem:LowWeightFracDecomposition} appear in~\cref{Sec:boosting}.

Finally, for (4) we use the following special case of Nibble with Reserves tailored to designs in~\cite{P25Survey}.

\begin{lem}[Corollary 2.4 in \cite{P25Survey}]\label{cor:NibbleReservesSpecific}
For every $q > r\ge 1,~\rho \in (0,1]$, there exists $\alpha\in (0,1)$ such that the following holds for all $n$ large enough. Let $G\subseteq K_n^r$ and $\mathcal{H}\subseteq \binom{G}{K_q^r}$ such that every edge of $G$ is in $\left(\rho \pm n^{-1/3}\right) \binom{n}{q-r}$ cliques of $\mathcal{H}$. If there exists $X\subseteq K_n^r\setminus G$ such that every $e\in G$ is in at least $n^{q-r-\alpha}$ $K_q^r$-cliques of $X\cup \{e\}$, then there exists a $K_q^r$-packing of $G\cup X$ that cover all edges of $G$. 
\end{lem}

\subsection{Proof Overview of the Hypergraph Nash-Williams' Omni-Absorber Theorem}\label{Subsec:PfOverviewAbsorbers}

To accomplish step (2) of our proof outline, we prove the following theorem. 

\begin{thm}[Hypergraph Nash-Williams' Refined Efficient Omni-Absorber]\label{thm:NWRefinedEfficientOmniAbsorber}
For every integer $r \ge 1$, there exists a real $c>0$ such that the following holds. For every integer $q >r$, there exists an integer $C\ge 1$ such that the following holds for large enough $n$:
\vskip.1in
\noindent If $X\subseteq  G\subseteq K^r_n$ with $\delta(G)\ge  \left(\ 1 - \frac{c}{\binom{q}{r-1}}\right)\cdot n $ and $\Delta(X) \le \frac{n}{C}$, then there exists a $C$-refined $K^r_q$-omni-absorber $A$ for $X$ in $G$ with $\Delta(A) \le C\cdot \max\left\{\Delta(X),~n^{1- 1/r}\cdot \log n\right\}$.    
\end{thm}
As discussed above, to prove~\cref{thm:NWRefinedEfficientOmniAbsorber}, we will use~\cref{thm:RefinedEfficientOmniAbsorber} as a black box, then re-embed fake-edges and private absorbers to ensure the edges used are present in $G$. Formalizing this idea requires the following definitions. 

\begin{definition}[Rooted Graph]
    Let $G$ be a hypergraph and $R$ be a subset of vertices from an arbitrary ground set. We say $G$ is \emph{$R$-rooted} (equivalently, \emph{rooted at $R$}) if $R \subseteq V(G)$ and $R$ is independent in $G$. 
\end{definition}

\begin{definition}[Embedding a Rooted Graph]
    Let $G$ be a hypergraph with $R \subseteq V(G)$ and let $H$ be an $R$-rooted hypergraph. An \emph{embedding} of $H$ \emph{into} $G$ is an edge-preserving map $\phi: V(H) \hookrightarrow V(G)$ such that $\phi(v) = v$ for all $v \in R$. 
\end{definition}

We also require that the gadgets we wish to embed satisfy the following natural condition.

\begin{definition}[Edge-Intersecting]
Let $J$ be a hypergraph. A $V(J)$-rooted hypergraph $H$ is \emph{$J$-edge-intersecting} if for all $e\in H$, there exists $f\in J$ such that $e\cap V(J) \subseteq f$.
\end{definition}

To re-embed the omni-absorber from~\cref{thm:RefinedEfficientOmniAbsorber}, we utilize the General Embedding Lemma from~\cite{EMeetsNW}. At a high level, it says that if (a) the number and size of the edge-intersecting, rooted gadgets we want to embed are relatively small, (b) the maximum degree of the subgraph induced by the root vertices is small, and (c) there are enough possible embeddings for each individual rooted gadget into $G$, then there exists an edge-disjoint embedding of our gadgets with low maximum degree. Since our fake-edge and absorber constructions are relatively small and their roots are in the refined omni-absorber from~\cref{thm:RefinedEfficientOmniAbsorber}, our main challenge is proving that the third property of many embeddings is satisfied. 

For an $r$-graph $H$ and $R\subseteq V(H)$, the \emph{degeneracy of $H$ rooted at $R$} is the smallest non-negative integer $d$ such that there exists an ordering $v_1,\ldots, v_{v(H)-|R|}$ of $V(H)\setminus S$ such that for all $i\in [v(H)-|R|]$, we have $|\{e\in E(H): v_i \in e \subseteq R\cup \{v_j: 1\le j \le i\}\}| \le d$.
Since fake-edges have rooted degeneracy at most $\binom{q-1}{r-1}$ while the $G$ in~\cref{thm:NWRefinedEfficientOmniAbsorber} has minimum degree strictly greater than $\left(1 - 1/\binom{q-1}{r-1}\right)\cdot n$, it is straightforward to show fake-edges have many possible embeddings in $G$ by embedding them one vertex at a time. However, the hypergraph absorbers constructed in~\cite{DKP25} have insufficient rooted degeneracy to establish Theorem~\ref{thm:NWRefinedEfficientOmniAbsorber} (indeed, using rooted degeneracy arguments alone, they would only embed in $G$ if $G$ has minimum degree at least $\left(1-\frac{c}{q^{2r-2}}\right)\cdot n$). Thus we will instead directly prove the following theorem via non-uniform Tur\'an theory. 

\begin{thm}\label{thm:CountingAbsorbers}
For every positive integer $r$, there exists a real $c >0$ such that the following holds. For every integer $q > r$, positive integer $t$, there exists a positive $n_0$ and a real $\gamma \in (0, 1]$ such that the following holds. Let $G$ be an $r$-graph on $n\ge n_0$ vertices with $\delta(G) \ge \left(1 - \frac{c}{\binom{q}{r-1}}\right)n$. If $L$ is a $K_q^r$-divisible subgraph of $G$ with $v(L) \le t$, then there exists an $L$-edge-intersecting $K_q^r$-absorber $A$ for $L$ such that there are at least $\gamma \cdot n^{|V(A)\setminus V(L)|}$ embeddings of $A$ into $G$.
\end{thm}

To prove~\cref{thm:CountingAbsorbers}, we first show the existence of $K_q^r$-absorbers that satisfy a partite property, before exploiting this property to prove the existence of many embeddings of each absorber into $G$. A hypergraph $H$ is \emph{$q$-partite} if there exists a partition of $V(H) = V_1 \cup \ldots \cup V_q$ such that for every $e \in H$, we have $V(e) \cap V_i \le 1$ for each $i\in [q]$. Since the embedding of root vertices is fixed, we do not require the roots to be partite, motivating the following definition. 

\begin{definition}[Rooted $q$-Partite]
Let $G$ be an $R$-rooted hypergraph. We say $G$ is \emph{$R$-rooted $q$-partite} if $G[V(G)\setminus R]$ is $q$-partite.
\end{definition}

The following lemma shows that rooted partite hypergraphs have many possible embeddings into a high minimum degree host hypergraph. 

\begin{lem}\label{lem:EmbedOneCounting}
For every pair of positive integers $r$ and $d$, there exists a real $c >0$ such that the following holds. For every integer $q > r$, positive integers $m$ and $t$, there exists a positive $n_0$ and a real $\gamma \in (0, 1]$ such that the following holds. Let $G$ be an $r$-graph on $n\ge n_0$ vertices with $\delta(G) \ge \left(1 - \frac{c}{\binom{q}{r-1}}\right)n$ and let $R \subseteq V(G)$ with $|R| \le d\cdot q$. If $H$ is an $R$-rooted, $q$-partite hypergraph with $v(H) \le t$, then there are at least $\gamma \cdot n^{|V(H) \setminus R|}$ embeddings of $H$ into $G$.
\end{lem}

To prove~\cref{thm:NWRefinedEfficientOmniAbsorber}, we utilize the absorber construction from Delcourt, Kelly, and the second author in~\cite{DKP25}, which are built by layering rooted $q$-partite gadgets. The partiteness of these gadgets is implicit in~\cite{DKP25}, but we nevertheless explicitly formalize this in~\cref{sec:Absorber} since we also need to show the existence of edge-intersecting absorbers to utilize the General Embedding Lemma. 

In order to prove~\cref{lem:EmbedOneCounting} and its generalization to layered partite hypergraphs, such as the absorbers from~\cite{DKP25}, we establish a non-uniform \Turan Theory as described in the next subsection. The relationship between a rooted Tur\'an theory as in~\cref{lem:EmbedOneCounting} and a non-uniform Tur\'an theory is as follows. Namely, the existence of one embedding of $H$ into $G$ rooted at $R$ would follow from finding a copy $H'$ of a truncated non-uniform version of $H$ (i.e.~an $i$ set is in $H'$ if and only if it extends to an edge with some $(r-i)$-set of the roots) in the non-uniform hypergraph $G'$ consisting of the $i$-sets $S$ for each $i \in [r]$ where $S$ extends to an edge with every set of $r-i$ roots. We note this is more than desired since finding such a copy in fact finds the truncated version of $H$ where each $i$-set forms an edge with \emph{all} $(r-i)$ subsets of the roots. It is important to note that the edge densities of the different uniformities will vary in $G'$ but are related to the minimum degree of $G$ (namely the edge-density of the $i$th uniformity of $G'$ would be at least $1-\frac{c}{q^{i-1}}$).

\subsection{Non-Uniform Tur\'an Theory}\label{Subsec:NonUnifTuran}

Here we establish a non-uniform \Turan Theory that encompasses these varying edge densities. We note that various authors over the years have sought to establish a Tur\'an theory for non-uniform hypergraphs (or equivalently simplicial complexes), which broadly fall into two categories as we now discuss, but these theories are not sufficient for our purposes. The first stream dates to the work of Frankl~\cite{F83} in 1983 and was first systematically studied by Conlon, Piga, and Sch{\"u}lke~\cite{CPS23} in 2023. However, in this stream the definition of Tur\'an number for an $r$-bounded hypergraph $F$ is the maximum number of edges of an $F$-free $r$-bounded hypergraph, which means that said numbers are dominated by edges of the largest uniformity. A second stream of results concern the Lubell function arising in the theory of posets (we refer the reader to the work of Johnston and Lu~\cite{JL14} in 2014 for a systematic study). In this stream, the Lubell function of an $r$-bounded hypergraph $F$ is defined as the maximum number of the sum of the edge densities of the different uniformities of an $F$-free $r$-bounded hypergraph. This scaling by density alleviates the issue of the largest uniformity dominating as in the first stream. However, we do not desire results about one single number but rather the whole space of densities which guarantee a copy of $F$. Thus we were led to develop our theory of a \Turan space for non-uniform hypergraphs as follows.

Recall that the \emph{\Turan number $ex(n, F)$} of an $r$-graph $F$ is the maximum number of edges in an $F$-free $r$-graph on $n$ vertices. Similarly, the \emph{\Turan density of $F$} is the limit of $ex(n, F)/\binom{n}{r}$ as $n$ goes to infinity.

For a hypergraph $G$, let $G^{(i)}$ denote the subhypergraph containing exactly the edges of uniformity $i$ in $G$. Recall that a hypergraph $G$ is \emph{$r$-bounded} if every edge contains at most $r$ vertices.

\begin{definition}[Tur\'an Space]       
Let $F$ be an $r$-bounded hypergraph. We say $(\alpha_1, \ldots, \alpha_r)\in [0,1]^r$ is a \emph{Tur\'an point of $F$} if for every $\varepsilon > 0$, there exists $n_0$ such that for all integers $n\ge n_0$ the following holds: if $G$ is a hypergraph on $n$ vertices such that $e(G^{(i)})\geq (\alpha_i+\varepsilon) \binom{n}{i}$ for all $i \in [r]$, then $G$ contains a copy of $F.$ The \emph{Tur{\'a}n space of $F$}, denoted $\TS(F)$, is the closure of the set of all Tur{\'a}n points of $F$.
\end{definition}

Discovered by \Erdos and Simonovits \cite{ES82}, the notion of supersaturation is a standard result in \Turan theory. Informally, it states that if a graph $G$ is above the \Turan density of a graph $F$, then $G$ contains not only one copy of $F$ but many. We prove that supersaturation holds in the non-uniform setting as follows (e.g.~generalizing the classical result such as Lemma 2.1 in \cite{K11}). 

\begin{lem}[Non-Uniform Supersaturation]\label{lem:supersaturation}
    For every $r$-bounded hypergraph $F$ and real $\varepsilon > 0$, there exists a real $\gamma > 0$ and integer $n_0$ such that for every $n \geq n_0$ the following holds: if $(\alpha_1, \ldots, \alpha_r) \in \TS(F)$ and $G$ is a hypergraph on $n$ vertices such that $e(G^{(i)}) \geq (\alpha_i + \varepsilon)\binom{n}{i}$ for every $i \in [r]$, then there exist at least $\gamma \cdot \binom{n}{v(F)}$ copies of $F$ in $G.$
\end{lem}

Recall that the \emph{$t$-blow-up} of hypergraph $F$, denoted by $F(t),$ is the graph obtained from $F$ by replacing every vertex $v \in V(F)$ with an independent set $U_v$ of size $t$ and every edge $e = \{v_1, \ldots, v_r\} \in E(F)$ with a complete $r$-partite $r$-uniform hypergraph spanned by $U_{v_1}, \ldots, U_{v_r}.$
We will prove the following theorem about blowing up in the non-uniform setting (e.g.~generalizing Theorem 2.2 in \cite{K11}).

\begin{lem}[Non-Uniform Blowing Up]\label{lem:BlowUp}
    Let $F$ be an $r$-bounded hypergraph and let $t$ be a positive integer. If $(\alpha_1, \ldots, \alpha_r) \in \TS(F),$ then $(\alpha_1, \ldots, \alpha_r) \in \TS(F(t)).$
\end{lem}

Finally, let $K_n^{[r]}$ denote the \textit{$r$-bounded non-uniform complete graph on $n$ vertices}, that is, the hypergraph with $V(K_n^{[r]}) = [n]$ and $E(K_n^{[r]}) = \bigcup_{i \in [r]} E(K_n^i)$. 
We aim to show that our minimum degree condition suffices to guarantee that our hypergraph contains a copy of $K_n^{[r]}$. To that end, we aim to prove there exists a point $(\alpha_1, \alpha_2, \ldots, \alpha_r) \in \TS(K_n^{[r]})$ where $\alpha_i$ is on the order of $1 - \frac{1}{\binom{q}{i-1}}$ for each $i \in [r]$. As mentioned in the introduction, the best known general upper bound for the Tur\'an number of $K_q^r$ is a result of de Caen~\cite{dC83} from 1983 with a bound $1 - \frac{1}{\binom{q-1}{r-1}}$. De Caen's proof generalized to hypergraphs the earlier work of Moon and Moser~\cite{MM65} in 1965 on graphs. It seemed non-trivial to us to attempt to generalize de Caen's work to the non-uniform setting. Instead as we only desire a result of the correct order, we look to the earlier (now often overlooked) work of Spencer~\cite{S72} from 1972 who proved an upper bound on the Tur\'an number of $1 - \frac{r!}{r^r}\left(\frac{r-1}{q-1}\right)^{1-r}$. Note that for large $q$ this is only a factor of $e$ worse than de Caen's bound. Spencer's proof, which only used the method of alterations, is much easier to generalize to our non-uniform setting as follows.

\begin{lem}[Non-Uniform Spencer's]\label{lem:NonUnifSpencer}
For every integer $r \ge 1$, there exists a real $c >0$ such that the following holds. Let $q > r$ be an integer. If $\alpha_i = 1 - \frac{c}{\binom{q}{i-1}}$ for every $i \in [r]$, then $(\alpha_1, \alpha_2, \ldots, \alpha_r)~\in~\TS(K_q^{[r]})$. 
\end{lem}

The proofs of our Non-Uniform Supersaturation, Non-Uniform Blowing-up, and Non-Uniform Spencer's appear in~\cref{Subsec:Spencer}. We then use these to prove~\cref{lem:EmbedOneCounting} in~\cref{Subsec:RootedTuran}.

\begin{remark}
We note for the interested reader that an easier version of~\cref{lem:NonUnifSpencer} with a worse bound of $1-\frac{c}{\binom{q}{i}}$ appeared in the work of Delcourt and the second author in their proof of the Existence Conjecture~\cite{DPI} which thus would yield a bound of $1-\frac{c}{q^r}$ in our main theorems when combined with the rest of our work.     
\end{remark}

In order to prove~\cref{thm:CountingAbsorbers}, we generalize~\cref{lem:EmbedOneCounting} to work for our absorbers constructed from layered rooted partite graphs. To that end, we require the following definition. 

\begin{definition}[$(d, q)$-Rooted Partite Degeneracy]\label{def:PartiteDegeneracy}
    Let $d$ and $q$ be positive integers. An $R$-rooted hypergraph $G$ is \emph{$(d, q)$-rooted partite degenerate} if there exists an $m$-partition $V(G)\setminus R =V_1 \cup \ldots \cup V_m$ and for each $i \in [m]$, there exists a $q$-partition $V_i = V_{i, 1} \cup \ldots \cup V_{i,q}$ such that the following holds. For each $i\in [m]$, let $E_i := \{e \in G : V(e) \cap V_i \neq \emptyset \text{ and } \forall j > i, V(e) \cap V_j = \emptyset\}$.
    \begin{itemize}
        \item  For each $i \in [m]$, every $e \in E_i$, and each $j \in [q]$, we have $|V(e) \cap V_{i, j}| \le 1$; and
        \item For each $i \in [m]$, we have $|\bigcup_{e \in E_i}V(e) \setminus V_i| \le d \cdot q$.
    \end{itemize}
\end{definition}

The following lemma proves that there are many embedding options for each hypergraph absorber into $G$, which we use to prove~\cref{thm:CountingAbsorbers}.

\begin{lem}\label{lem:EmbedDegenPartite}
For every pair of positive integers $r$ and $d$, there exists a real $c >0$ such that the following holds. For every integer $q > r$, positive integers $m$ and $t$, there exists a positive $n_0$ and a real $\gamma \in (0, 1]$ such that the following holds. Let $G$ be an $r$ graph on $n\ge n_0$ vertices with $\delta(G) \ge \left(1 - \frac{c}{\binom{q}{r-1}} \right)n$ and let $R \subseteq V(G)$ with $|R| \le t$. If $H$ is an $R$-rooted, $(d, q)$-rooted partite degenerate hypergraph with respect to the partition $V(H)\setminus R = V_1 \cup \ldots \cup V_m$ and $|V_i| \le t$ for all $i \in [m],$ then there are at least $\gamma \cdot n^{|V(H) \setminus R|}$ embeddings of $H$ into $G$.
\end{lem}

We also prove~\cref{lem:EmbedDegenPartite} in~\cref{Subsec:RootedTuran}. In light of~\cref{lem:EmbedDegenPartite}, in order to prove~\cref{thm:CountingAbsorbers}, it suffices to prove the existence of rooted partite degenerate absorbers, as follows. 

\begin{thm}\label{thm:PartiteAbsorberExistence}
     For every integers $q > r\ge 1$, if $L$ is a $K_q^r$-divisible $r$-graph, then there exists a $(2, q)$-rooted partite degenerate $L$-edge-intersecting $K_q^r$-absorber for $L$.
\end{thm}

In~\cref{sec:Absorber} we utilize the absorber construction from~\cite{DKP25} to prove~\cref{thm:PartiteAbsorberExistence}, before proving~\cref{thm:NWRefinedEfficientOmniAbsorber}.

\subsection{Outline of Remainder of Paper}
In~\cref{sec:Turan}, we establish our non-uniform \Turan Theory before proving embedding results for rooted, partite degenerate hypergraphs (\cref{lem:EmbedOneCounting} and~\cref{lem:EmbedDegenPartite}). 
In~\cref{sec:Absorber}, we include the definitions and proofs related to the existence of rooted partite degenerate edge-intersecting $K_q^r$-absorbers, before deriving the existence of absorbers with many embeddings into a high minimum degree hypergraph $G$ (\cref{thm:CountingAbsorbers}).
In~\cref{Sec:Omni}, we accomplish step (2) of the proof outline by proving the Hypergraph Nash-Williams' Refined Efficient Omni-Absorber Theorem (\cref{thm:NWRefinedEfficientOmniAbsorber}).
In~\cref{Sec:boosting}, we accomplish step (3) of the proof outline by first proving the existence of a low-weight fractional decomposition of $G$ (\cref{lem:LowWeightFracDecomposition}), then deriving the Hypergraph Nash-Williams' Boosting Theorem (\cref{thm:NWRegBoost}).
Finally, in~\cref{sec:Finish}, we first accomplish step (1) of the proof outline by proving the Hypergraph High Minimum Degree Reserves Lemma (\cref{lem:highmindegreserves}). Then we derive~\cref{thm:main} by verifying that the above results work together as described in the proof outline. In~\cref{Sec:Further}, we discuss further directions of research. The appendix contains proofs implied by previous results but which we include for completeness, such as an alternate proof of the General Embedding Lemma from~\cite{EMeetsNW}.

\section{Non-Uniform Tur\'an Theory}\label{sec:Turan}
In this section, we establish the necessary non-uniform \Turan theory results for our proof of the existence of efficient omni-absorbers in our high-minimum degree hypergraph setting. First, in~\cref{Subsec:Spencer}, we generalize several classical results in hypergraph Tur\'an theory (see the survey of Keevash~\cite{K11}) to the setting of non-uniform \Turan Spaces, which we then employ in~\cref{Subsec:RootedTuran} to prove~\cref{lem:EmbedOneCounting} and~\cref{lem:EmbedDegenPartite}, our rooted partite degenerate embedding lemmas. 

\subsection{Non-Uniform Supersaturation, Blow-up, and Spencer's Lemmas}\label{Subsec:Spencer}

As described in~\cref{Subsec:NonUnifTuran}, in order to prove an embedding lemma for absorbers in our hypergraph Nash-Williams' setting, we first establish a Non-Uniform Supersaturation Lemma (\cref{lem:supersaturation}), a Non-Uniform Blowing Up Lemma (\cref{lem:BlowUp}), and a Non-Uniform Spencer's Lemma (\cref{lem:NonUnifSpencer}).

We first focus our attention on supersaturation. In~\cite{K11}, Keevash proved showed that, for an appropriately chosen $k$, an $r$-graph $G$ with edge-density above the \Turan number of $F$ contains many $k$-sets that induce a subgraph with a similarly high edge-density, each of which contains a copy of $F$. To prove our non-uniform supersaturation lemma, we use a similar tactic on each uniformity, then argue that there are many $k$-sets that induce a subgraph with high edge-density in every uniformity. To achieve this, we require a stronger bound for the number of $k$-subsets that induce a high edge-density subgraph in a given uniformity. Inspired by the footnote appearing in Keevash's proof of Lemma 2.1 in \cite{K11}, we use the following result to prove~\cref{lem:supersaturation}.

\begin{lem}\label{lem:e(G[S])large}
    For every positive integer $r$ and reals $ \alpha > \beta > 0$, there exists a positive integer $k_0$ such that if $k \geq k_0$ and $n \geq 2k^2$, then the following holds: If $G$ is a $r$-graph on $n$ vertices such that $e(G) \geq \alpha \binom{n}{r}$, then there are at most $\frac{1}{\sqrt{k}}\binom{n}{k}$ $k$-subsets $S$ of $V(G)$ such that $e(G[S]) < \beta \binom{k}{r}$. 
\end{lem}

\begin{remark}
We note that Lang's Inheritance Lemma \cite{lang2023tiling}, which we use a special case of later in~\cref{Sec:boosting} as Lemma~\ref{lem:LangInheritance}, provides a similar result phrased in terms of a minimum $i$th-degree condition for subgraphs containing a fixed set of `roots'. However, the statement of Lang's lemma does not technically include the case of the $0th$-degree, which is the edge-density condition (though this case is admittedly much simpler and would follow from his proof). We note that Ferber and Kwan~\cite{FB22} had earlier provided a proof of the simpler non-rooted version of the inheritance lemma (though they did not specifically state the range of $d$ their proof works for). In~\cite{K11}, Keevash mentions that the above lemma would follow by standard concentration inequalities (namely second moment method, Talagrand's inequality, or the Azuma-Hoeffding inequality all suffice) but we include a proof in the appendix for completeness.    
\end{remark}

Now we are ready to prove~\cref{lem:supersaturation}.

\begin{proof}[Proof of Lemma \ref{lem:supersaturation}]
Let $F$ be an $r$-bounded hypergraph, and let $(\alpha_1, \ldots, \alpha_r)$ be a real vector such that $(\alpha_1, \ldots, \alpha_r) \in \TS(F)$. Let $\eps > 0$ be a real number. For each $i \in [r]$, let $k_i$ be an integer such that~\cref{lem:e(G[S])large} holds for uniformity $i$, $\alpha = \alpha_i$, and $\beta = \left(\alpha_i + \frac{\eps}{2}\right)$. By the definition of \Turan space, there exists a positive integer $k_0$ such that every hypergraph $H$ on at least $k_0$ vertices with $e(H^{(i)}) \geq (\alpha_i+\frac{\eps}{2})\binom{k}{i}$ for every $i \in [r]$ contains $F$ as a subgraph. Let $k \ge \max\{k_0, k_1, \ldots, k_r, 4r^2\}$ be an integer, $\gamma := \frac{1}{2\binom{k}{v(F)}}$, and $n \ge 2k^2$. Let $G$ be a hypergraph on $n$ vertices such that $e(G^{(i)}) \geq (\alpha_i + \varepsilon)\binom{n}{i}$ for every $i \in [r]$. 

By Lemma~\ref{lem:e(G[S])large}, for each $i\in [r]$, there are at most $\frac{1}{\sqrt{k}} \binom{n}{k}$ $k$-subsets $S$ of $V(G)$ with $e(G^{(i)}[S]) < (\alpha_i + \frac{\varepsilon}{2})\binom{k}{i}$. Hence, there are at least $\left(1 - \frac{r}{\sqrt{k}} \right) \binom{n}{k}$ $k$-subsets $S$ of $V(G)$ such that $e(G^{(i)}[S]) \geq (\alpha_i + \frac{\varepsilon}{2})\binom{k}{i}$ for every $i\in [r]$. Since $k \ge 4r^2$, there are at least $\frac{1}{2} \binom{n}{k}$ such sets. Also, each such set $S$ contains $F$ as a subgraph because $k \ge k_0$. Therefore, by accounting for the number of $k$-sets that contain the same copy of $F$, we see that the number of copies of $F$ in $G$ is at least $\frac{\frac{1}{2} \binom{n}{k} }{\binom{n - v(F)}{k-v(F)}} = \frac{1}{2 \binom{k}{v(F)}} \cdot \binom{n}{v(F)} = \gamma \cdot \binom{n}{v(F)}$, as desired. 
\end{proof}

We now turn our attention to Lemma \ref{lem:BlowUp}, the Non-Uniform Blowing Up Lemma. Note that by substituting Lemma \ref{lem:supersaturation} for the Supersaturation Lemma employed by Keevash in the proof of Theorem~2.2 in~\cite{K11}, Keevash's proof directly generalizes to the non-uniform setting. We include the argument here for completeness. 

We require two classical results from \Turan Theory and Ramsey Theory in the proof of Lemma \ref{lem:BlowUp}. First, \Erdos~\cite{E64} famously calculated the \Turan density of partite hypergraphs, resulting in the following statement. 

\begin{thm}[\Erdos\cite{E64}]\label{thm:ErdosTuranDensity}
For every partite hypergraph $H$, the \Turan density of $H$ is zero.
\end{thm}

Next, Rado~\cite{R54} generalized the notorious Ramsey's Theorem to the partite hypergraph setting. In particular, for integers $q \ge r$, let $K_{q*n}^r$ denote the complete $q$-partite $r$-graph with $n$ vertices in each part (that is, the hypergraph whose vertex set consists of $q$ parts, each of size $n$, and whose edge set consists of all $r$-sets with at most one vertex in each part). For an integer $c$ and a $q$-partite $r$-graph $H$, let $R_c(H)$ denote the smallest integer $n$ such that any coloring of the edges of $K_{q * n}^r$ with $c$ colors contains a monochromatic copy of $H$, if such an integer exists. Otherwise, let $R_c(H) = \infty$. Rado proved that these $q$-partite $r$-uniform Ramsey numbers always exist. 

\begin{thm}[Rado \cite{R54}]\label{Thm:PartiteHGRamsey}
 For every integers $c$ and $q \ge r$ and every $q$-partite, $r$-graph $H$, the value $R_c(H)$ is finite. 
\end{thm}

Finally, we recall the definition of the \emph{$F$-design hypergraph of $G$}, denoted ${\rm Design}(G,F)$, as the hypergraph $\mathcal{D}$ with $V(\mathcal{D}):=E(G)$ and $E(\mathcal{D}) := \{S\subseteq E(G): S \text{ is isomorphic to } F\}$. We are now ready to prove~\cref{lem:BlowUp}.

\begin{proof}[Proof of Lemma \ref{lem:BlowUp}]
    Let $F$ be an $r$-bounded hypergraph, and let $(\alpha_1, \ldots, \alpha_r)$ be a real vector such that $(\alpha_1, \ldots, \alpha_r) \in \TS(F)$. Fix a positive integer $t$, and let $t':= R_{v(F)!}(K_{v(F) * t}^{v(F)})$. Note that $t'$ is finite by~\cref{Thm:PartiteHGRamsey}. Fix $\varepsilon >0$ and let $n$ be as large as necessary throughout the proof. Let $G$ be a hypergraph on $n$ vertices such that $e(G^{(i)}) \geq (\alpha_i + \varepsilon)\binom{n}{i}$ for every $i \in [r]$. Let $H:={\rm Design}(G,F)$. By Lemma \ref{lem:supersaturation}, there exists $\gamma > 0$ such that $G$ contains at least $\gamma \cdot \binom{n}{v(F)}$ copies of $F.$ Thus, the edge-density of $H$ is positive, and so, by Theorem~\ref{thm:ErdosTuranDensity}, there exists a copy $K$ of $K_{v(F) * t'}^{v(F)}$ in $H$ with vertex partition $V(K) = V_1 \cup \ldots \cup V_{v(F)}.$ In order to employ~\cref{Thm:PartiteHGRamsey}, we color the edges of $K$ with $v(F)!$ colors according to how the vertices of a copy of $F$ intersect the parts of $K$. To that end, for every $e \in E(H)$, let $F_e$ be the copy of $F$ is $G$ corresponding to edge $e$. Let $\phi:[v(F)!] \rightarrow E(K)$ be a function such that for every $e, f \in E(K)$, we have that $\phi(e) = \phi(f)$ if and only if $V(F_e) \cap V_i = V(F_f)\cap V_i$ for every $i\in [t']$. By the choice of $t'$, there exists a monochromatic copy of $K_{v(F)*t}^{v(F)}$ in $H$. Equivalently, there exists a copy of $F(t)$ in $G$, as desired.
\end{proof}

We are now ready to prove~\cref{lem:NonUnifSpencer}. Our proof is a generalization of Spencer's original argument in~\cite{S72} using probabilistic alteration.

\begin{proof}[Proof of Lemma \ref{lem:NonUnifSpencer}]
Let $n$ be an integer, taken to be as large as necessary. Let $c:= \frac{1}{r\cdot 6^r}$. In accordance with Spencer's original argument, we will work in the complementary setting and show that if $G$ is an $r$-bounded graph on $n$ vertices such that $e(G^{(i)}) \leq \frac{c}{\binom{q}{i-1}} \binom{n}{i}$ for each $i \in [r]$, then there exists an independent set of size at least $q$ in $G$. Let $p := 2q/n$. For the purposes of alteration, let $f$ be a choice function that assigns to each nonempty subset $S\subseteq V(H)$ an element $f(S) \in S$. Let $A$ be a random subset of $V(G)$ such that each vertex is chosen independently with probability $p$. Likewise, define $A^*:= A - \bigcup _{e \in E(G[A])}f(e)$. Observe that $A^*$ is an independent set. By linearity of expectation, we have that 
\begin{align*}
    \E[~|A^*|~] &\ge \E[~|A|~] - \sum_{i \in [r]}\E[e(G^{(i)}[A])] \\
    & = pn -\sum_{i \in [r]} p^i e(G^{(i)}) \\
    & \ge pn - \sum_{i \in [r]} \frac{c \cdot p^i}{\binom{q}{i-1}}\binom{n}{i} \\
    & \ge pn - \sum_{i \in [r]} \frac{c  \cdot e^i\cdot (i-1)^{(i-1)}\cdot (pn)^i}{i^i \cdot q^{i-1}} \\
    & \ge 2q - \sum_{i \in [r]} c \cdot e^i\cdot  2^i \cdot q \\
    & \ge 2q - r \cdot c \cdot (2e)^r \cdot q \\
    &\ge q,
\end{align*}
where the last inequality holds since $c \le\frac{1}{r \cdot (2e)^r}$. Therefore, there exists a set $A$ such that $|A^*| \ge \E[~|A^*|~] \ge q$, as desired.
\end{proof}

\subsection{Rooted Non-Uniform \Turan Theory}\label{Subsec:RootedTuran}

In this section, we prove~\cref{lem:EmbedDegenPartite}. We first establish~\cref{lem:EmbedOneCounting} which proves the existence of many embeddings of one rooted partite $r$-graph $H$ in a high minimum degree host $G$. In the proof, we use the non-uniform \Turan theory we established in~\cref{Subsec:NonUnifTuran} applied to the $r$-bounded hypergraph consisting of the $i$-sets $S$ for each $i \in [r]$ where $S$ extends to an edge with every set of $r-i$ roots. (This could be construed as ``projecting'' out the roots in the compliment of $G$, using a definition of projection similar to the one in~\cite{DPII}.) 

\begin{lateproof}{lem:EmbedOneCounting}
Let $c'$ be as in~\cref{lem:NonUnifSpencer} for $r$ and let $c:=\frac{c'}{2(dr)^r}.$ Let $H'$ be the $r$-bounded hypergraph with $V(H') = V(H) \setminus R$ and $E(H') = \{e \setminus R : e \in E(H)\}$. Note that since $H$ is $R$-rooted $q$-partite, we have that $H'$ is $q$-partite. Similarly, let $G'$ be the $r$-bounded hypergraph with $V(G') := V(G)$ and, for each $i \in [r]$, 
$$E(G'^{(i)}) := \left\{e \in \binom{V(G)\setminus R}{i} : e \cup f \in E(G) \text{ for all } f \in \binom{R}{r-i}\right\}.$$

Observe that for every distinct embedding of $H'$ into $G'$, there exists a distinct embedding of $H$ rooted at $R$ into $G$. Thus, it suffices to show there exists at least $\gamma \cdot n^{|V(H) \setminus R|}$ embeddings of $H'$ into $G'$ for some $\gamma \in (0, 1]$. 

To that end, for each $i \in [r]$, let $\alpha_i := 1 - \frac{c'}{\binom{q}{i-1}}$. Recall that $(\alpha_1, \alpha_2, \ldots, \alpha_r) \in \TS(K_q^{[r]})$ by~\cref{lem:NonUnifSpencer}. Since $c' >0$, we have $(\alpha_1, \alpha_2, \ldots, \alpha_r) \in \TS(K_q^{[r]}(t))$ by~\cref{lem:BlowUp}. 

Next, let $\eps':= \frac{c}{\binom{q}{r-1}}$ and fix $i \in [r]$. We aim to show that $e(G'^{(i)}) \ge (\alpha_i + \eps')\binom{n}{i}$. To that end, let $M_i:= E(K_n^i) - E(G'^{(i)})$. Recall that $\delta(G) \ge \left(1 - \frac{c}{\binom{q}{r-1}} \right)n$, so for every choice of $r-i$ vertices in $R$, and every choice of $i-1$ in $V(G) \setminus R$, there are at most $\frac{c \cdot n}{\binom{q}{r-1}}$ edges in $M_i$ containing these $i-1$ vertices. Hence,
\begin{align}
    |M_i| \le \frac{1}{i}\binom{|R|}{r-i} \binom{n - |R|}{i - 1}{\frac{c \cdot n}{\binom{q}{r-1}}} \le \frac{1}{i}\binom{dq}{r-i} \binom{n - dq}{i - 1}{\frac{c \cdot n}{\binom{q}{r-1}}}, \label{line:M_i}
\end{align}
where the second inequality follows since $|R| \le dq$ and $n$ is large enough. Also note that by the standard binomial bounds $\frac{n^k}{k^k} \le \binom{n}{k} \le n^k$, we have
\[\eps' \binom{q}{i-1} \le \frac{c \cdot q^{i-1} \cdot (r-1)^{r-1}}{q^{r-1}} \le \frac{c \cdot r^{r-1}}{q^{r-i}} \le c \cdot (dr)^{r-1},\]
which implies
\[c \cdot (dr)^{r-1} \le 2 \cdot c \cdot (dr)^{r-1} - \eps' \binom{q}{i-1}.\]
Combined with standard binomial bounds, this implies
\begin{align}
    \frac{c \cdot \binom{dq}{r-i}}{\binom{q}{r-1}} \le \frac{c \cdot d^{r-i}\cdot q^{r-i} \cdot (r-1)^{r-1}}{q^{r-1}} \le \frac{c \cdot (dr)^{r-1}}{\binom{q}{i-1}} \le \frac{2 \cdot c \cdot (dr)^{r-1}}{\binom{q}{i-1}} - \eps'. \label{line:eps}
\end{align}
Thus, 
\begin{align*}
    e(G'^{(i)}) = e(K_n^i) - |M_i| &\ge \binom{n}{i} - \frac{1}{i} \binom{dq}{r-i} \binom{n - dq}{i - 1}{\frac{c n}{\binom{q}{r-1}}}\\
    & \ge \binom{n}{i} - \frac{n}{i}\binom{n - dq}{i - 1}\left(\frac{2 \cdot c \cdot (dr)^{r-1}}{\binom{q}{i-1}} - \eps'\right) \\
    & \ge \binom{n}{i} - \frac{n}{i}\binom{n-1}{i-1}\left(\frac{c'}{\binom{q}{i-1}} - \eps'\right)\\
    & = \binom{n}{i} - \binom{n}{i}\left(\frac{c'}{\binom{q}{i-1}} - \eps'\right) \\
    & = \left(\alpha_i + \eps' \right)\binom{n}{i},
\end{align*}
where the first line follows from~\cref{line:M_i}, the second from ~\cref{line:eps}, and the third since $n$ is large enough. 
Hence, by~\cref{lem:supersaturation}, there exists a real integer $\gamma' > 0$ such that there are at least $\gamma'  \cdot \binom{v(G')}{v(K_q^{[r]}(t))} = \gamma'  \cdot \binom{n}{q \cdot t}$ copies of $K_q^{[r]}(t)$ in $G'$. 

Note that since $H'$ is $q$-partite and $v(H') \le v(H) \le t$, each copy of $K_n^{[r]}(t)$ contains at least one embedding of $H'$. Therefore, there are at least
$$\gamma'  \cdot \binom{n}{q \cdot t} \ge \gamma' \cdot \frac{1}{(q \cdot t)^{(q\cdot t)}} \cdot n^{v(H')} = \gamma' \cdot \frac{1}{(q \cdot t)^{(q\cdot t)}} \cdot n^{|V(H) \setminus R|}$$
embeddings of $H'$ into $G'$. By taking $\gamma := \frac{\gamma'}{(q \cdot t)^{(q\cdot t)}}$, we see that there are at least $\gamma \cdot n^{|V(H) \setminus R|}$ embeddings of $H$ into $G$, as desired.
\end{lateproof}

Now we are prepared to prove~\cref{lem:EmbedDegenPartite}.

\begin{proof}
Let $c$ be as in~\cref{lem:EmbedOneCounting} for $r$ and $d$. We proceed by induction on $m$ and note that Lemma~\ref{lem:EmbedOneCounting} provides a base case. 

Let $H_1:= H - H[V_m]$. Assume there exists at least $\gamma_1 \cdot n^{|V(H_1) \setminus R|}$ embeddings of $H_1$ into $G$ for some $\gamma_1 \in (0, 1]$. Also let $R_2: = \{v \in V(H) \setminus V_m : \exists e \in E(H) \text{ with } v \in V(e) \text{ and } V(e) \cap V_m \neq \emptyset\}.$ By the $(d, q)$-rooted partite degeneracy of $H$, we know $|R_2| \le d \cdot q$. Let $H_2:= H[V_m \cup R_2]$. For a fixed embedding $\phi$ of $H_1$ into $G$, let $R_{\phi}= \{v \in V(G) : \exists u \in V(H_1)\setminus R_2 \text{ with } \phi(u) = v \}$ and let $G_{\phi}:= G - G[R_{\phi}]$. Note that $|R_{\phi}| \le |R| + \sum_{i =1}^{m-1}|V_i| \le m \cdot t$. Since $n$ is large enough, observe that
\[\delta(G_{\phi}) \ge \delta(G) - |R_{\phi}| \ge \left(1 - \frac{c}{\binom{q}{r-1}}+ \eps \right)n - m\cdot t \ge \left(1 - \frac{c}{\binom{q}{r-1}} +\frac{\eps}{2} \right)(n- m\cdot t) = \left(1 - \frac{c}{\binom{q}{r-1}}+\frac{\eps}{2} \right)v(G_{\phi}). \]
Thus, by Lemma~\ref{lem:EmbedOneCounting}, there are at least $\gamma_2 \cdot n^{|V_m\setminus R_2|}$ embeddings of $H_2$ into $G_{\phi}$ for some $\gamma_2 \in (0, 1]$. Since the concatenation of an embedding $\phi$ of $H_1$ into $G$ and an embedding of $H_2$ into $G_{\phi}$ is an embedding of $H$ into $G$, by taking $\gamma := \gamma_1 \cdot \gamma_2$, we get that there are at least $\gamma_1 \cdot n^{|V(H_1) \setminus R|} \cdot \gamma_2 \cdot n^{|V_m\setminus R_2|} = \gamma \cdot n^{|V(H) \setminus R|}$ embeddings of $H$ into $G$, as desired.
\end{proof}

\section{Counting Absorbers}\label{sec:Absorber}
In this section, we prove~\cref{thm:CountingAbsorbers}. Observe that, having already established~\cref{lem:EmbedDegenPartite}, it suffices to show the existence of rooted partite degenerate edge-intersecting $K_q^r$-absorbers (\cref{thm:PartiteAbsorberExistence}). To accomplish this, we modify the $K_q^r$-absorber construction from Delcourt, Kelly, and the second author in~\cite{DKP25}, which are built by layering \textit{boosters} and \textit{hinges} on an initial \textit{integral decomposition.} At a high level, boosters are non-trivial clique absorbers that allow us to avoid using a particular clique in a decomposition, and hinges are structures that allow us to switch between which clique to use to decompose a given edge. As mentioned in the second author's survey~\cite{P25Survey}, in the absorber construction from~\cite{DKP25}, the gadgets are layered in an edge-intersecting way, so, to ensure the final absorber construction is edge-intersecting, it suffices to prove that the initial integral decomposition satisfies a similar property. The suggested solution in~\cite{P25Survey} is to build such an integral decomposition by induction on uniformity, similar to the Cover Down Lemma from Glock, \Kuhn, Lo, and Osthus~\cite{GKLO16}. We provide a rigorous proof of this in~\cref{subsec:EdgeIntersecting}. Likewise, the layering technique lends itself well to the definition of rooted partite degenerate, and thus, to prove our final absorber is rooted partite degenerate, it suffices to use $q$-partite boosters and hinges. This property is implicit in the construction from~\cite{DKP25}, but we codify this in~\cref{subsec:Boosters&Hinges}, before proving~\cref{thm:PartiteAbsorberExistence} and~\cref{thm:CountingAbsorbers} in~\cref{subsec:EdgeIntAbsorbers}.

\subsection{Edge-Intersecting Integral Decompositions}\label{subsec:EdgeIntersecting}

We first establish the necessary definitions and theorems related to the edge-intersecting integral decomposition our absorber is built upon.

\begin{definition}[Integral $K_q^r$-valuation]
An \emph{integral $K_q^r$-valuation} $\Phi\in \mathbb{Z}^{\binom{J}{K_q^r}}$  of an $r$-graph $J$ is an assignment of integers $\Phi(Q)$ to the $K_q^r$-cliques $Q$ of $J$. For $e\in J$, we define $\partial \Phi (e):= \sum_{Q\in \binom{J}{K_q^r}: e\in Q} \Phi(Q)$.
\end{definition}

\begin{definition}[Integral Decomposition]
An \emph{integral $K_q^r$-decomposition} $\Phi$ of an $r$-graph $G$ is an integral $K_q^r$-valuation of $K_{V(G)}^r$ such that $\partial \Phi(e) = +1$ for all $e\in G$ and $\partial \Phi(e) = 0$ for all $e\in K_{V(G)}^r\setminus G$.
\end{definition}

For the purposes of induction in the proof of the existence of integral $K_q^r$-decompositions, we require the following definition. An \textit{integral hypergraph} is a hypergraph $L$ equipped with an integral valuation of its edges $\Psi\subseteq \mathbb{Z}^{e(L)}$. We let $\Psi(e)$ denote the value of edge $e$ in $L$. The notions of $K_q^r$-divisibility and integral $K_q^r$-decomposability naturally extend to integral $r$-graphs.

In 1973, Graver and Jurkat \cite{GJ73} and, independently, Wilson \cite{W73} showed the existence of integral $K_q^r$-decompositions in large enough $K_q^r$-divisible $r$-graphs. Wilson's proof also holds for integral hypergraphs, as Keevash remarks in~\cite{KeevashShortEC}.

\begin{thm}[Wilson~\cite{W73}]\label{thm:Integral}
Let $q > r\ge 1$ be integers. If $L$ is a $K_q^r$-divisible integral $r$-graph with $v(L)\ge q+r$, then there exists an integral $K_q^r$-decomposition of $L$.
\end{thm}

As discussed above, we require our integral decomposition to satisfy an edge-intersecting property:

\begin{definition}[Edge-Intersecting Integral $K_q^r$-Valuation]
Let $L\subseteq K_n^r$ (not necessarily spanning). An integral $K_q^r$-valuation $\Phi$ of $K_n^r$ is \emph{$L$-edge-intersecting} if for every clique $Q\in \Phi$, there exists $f\in L$ such that $V(Q)\cap V(L) \subseteq f$.
\end{definition}

We also require the following definitions. For an integral $r$-graph $L$ and $S\subseteq V(L)$, the \emph{cone of $S$} is the set $\{e\in L: S\subseteq e\}$ and \emph{the link hypergraph of $S$} is $L(S) := \{ e\setminus S: e \text{ is in the cone of } S \}$. 

To prove the existence of edge-intersecting $K_q^r$-absorbers, we first prove the following generalization of Theorem~\ref{thm:Integral}.

\begin{thm}\label{thm:EdgeIntersectingIntegral}
Let $q > r\ge 1$ be integers. If $L$ is a $K_q^r$-divisible integral hypergraph with integral valuation $\Psi \in \mathbb{Z}^{e(L)}$, then there exists an $L$-edge-intersecting integral $K_q^r$-decomposition $\Phi$ of $L$ inside $K_{v(L)+q+r}^r$.
\end{thm}

The proof of Theorem~\ref{thm:EdgeIntersectingIntegral} proceeds by constructing an edge-intersecting integral decomposition $\Phi$ iteratively where the current remainder ($R:= L-\Phi$) has progressively smaller intersection with $V(L)$; in particular for each $i\in [r]_0$, we apply the Integral Decomposition Theorem (Theorem~\ref{thm:Integral}) to every set $S\subseteq \binom{V(L)}{r-i}$ with $R(S)\ne \emptyset$ to decompose the edges of $R$ containing $S$. The new cliques will be edge-intersecting since all such sets $S$ will be contained in some edge of $L$.

First we require the following proposition. 

\begin{proposition}\label{Prop:IntegralCone}
Let $\Psi\subseteq \mathbb{Z}^{\binom{[n]}{r}}$ be a $K_q^r$-divisible integral $r$-graph where $n\ge q+r$. If $S\subseteq [n]$ with $|S|\le r$, then there exists an integral $K_q^r$-valuation $\Phi$ of $K_n^r$ such that $\partial \Phi(e) = \Psi(e)$ for all $e\in \binom{[n]}{r}$ with $S\subseteq e$. 
\end{proposition}
\begin{proof}
Let $i:=|S|$. Since $\Psi$ is $K_q^r$-divisible, we find that $\Psi(S)$ is $K_{q-i}^{r-i}$-divisible. By Theorem~\ref{thm:Integral} as $n-i\ge (q-i)+(r-i)$, there exists a $K_{q-i}^{r-i}$-decomposition $\Phi'$ of $\Psi(S)$. For each $P\in \binom{[n]\setminus S}{q-i}$, we let $\Phi(P\cup S) := \Phi'(P)$. Then $\Phi$ is as desired.
\end{proof}

We are now prepared to prove Theorem~\ref{thm:EdgeIntersectingIntegral}.

\begin{proof}[Proof of Theorem~\ref{thm:EdgeIntersectingIntegral}]
We claim that for each $i\in [r]_0$, there exists 
an $L$-edge-intersecting integral $K_q^r$-valuation $\Phi_i$ of $G:=K_{v(L)+q+r}^r$ such that $\partial \Phi(e) =\Psi(e)$ for all $e\in L$ and $\partial\Phi(e) = 0$ for all $e \in G\setminus L$ with $|V(e)\cap V(L)|\ge r-i$. Note then that the theorem follows with $\Phi_r$ as desired.

We proceed by induction on $i$. Let $X := V(G)\setminus V(L)$. First suppose $i=0$. Let $X'\subseteq X$ with $|X'|=q-r$. Let $\Phi_0$ be the integral valuation such that $\Phi_0(V(e)\cup X')=\Psi(e)$ for every $e \in L$ and $0$ otherwise. Then $\Phi_0$ is as claimed.

So we assume $i> 1$. Let $\mc{S} := \bigcup_{e\in L} \binom{V(e)}{r-i}$. By Proposition~\ref{Prop:IntegralCone} as $|X|\ge q+r$, we have that for each $S\in \mc{S}$ there exists an integral $K_q^r$-valuation $\Phi_S$ of the cone of $S\cup X$ in $G$ such that $\partial \Phi_S(f) = -\Phi_{i-1}(f)$ for all $f\in \binom{V(e)\cup X}{r}$ with $S\subseteq f$. Then $\Phi_i:= \Phi_{i-1} + \sum_{S\in \mc{S}} \Phi_S$ is as claimed.
\end{proof}

\subsection{$q$-partite Boosters and Hinges}\label{subsec:Boosters&Hinges}

Here we show the existence of the partite version of the objects used to build a $K_q^r$-absorber. For a $K_q^r$-decomposition $\mc{Q}$ of an $r$-graph $G$ and edge $e\in G$, we use $\mc{Q}[e]$ to denote the (unique) $K_q^r$-clique in $\mc{Q}$ containing $e$.

\begin{definition}[Booster]\label{def:booster}
A \emph{$K_q^r$-booster} for a $K_q^r$-clique $S$ is an $r$-graph $B$ such that $B$ is edge-disjoint from $S$, $B$ has a $K_q^r$-decomposition $\off (B)$, and $B\cup S$ has a $K_q^r$-decomposition $\on (B)$ such that $S\not\in \on(B)$. 
\end{definition}

\begin{definition}[Orthogonal Booster]
A $K_q^r$-booster $B$ for $S$ is \emph{orthogonal} if for all distinct $e, f\in S$, we have $\on(B)[e]\ne \on(B)[f]$.
\end{definition}

Since we require our absorber to be partite, the following is a simple observation that we will refer to in subsequent proofs.

\begin{remark}\label{rem:partite}
    If $H_1$ and $H_2$ are $q$-partite $r$-graphs such that $H_1[V(H_1)\cap V(H_2)] = H_2[V(H_1)\cap V(H_2)]$, then $H_1 \cup H_2$ is also $q$-partite.
\end{remark}

Next, we show the existence of partite boosters and partite orthogonal boosters. 
Recall that $K_{q*n}^r$ denotes the complete $q$-partite $r$-graph with $n$ vertices in each part. The following is the partite analog of Lemma~2.3 from~\cite{DKP25}.

\begin{lem}\label{lem:Booster}
Let $q > r\ge 1$ be integers. If $S$ is a $K^r_q$-clique, then there exists a $K_q^r$-booster $B$ for $S$ such that $B\cup S$ is $q$-partite.
\end{lem}

\begin{proof}
Let $n$ be a prime such that $2q-r < n < 2(2q-r)$ (such exists by Bertrand's postulate, first proved in 1852 by Chebyshev~\cite{C52}). We may assume $S$ is a copy of $K_q^r$ in $K_{q*n}^r$. We prove that $K_{q*n}^r\setminus S$ is a $q$-partite $K_q^r$-booster for $S$. Since $K_{q*n}^r$ is $q$-partite by construction, it suffices to show that $K_{q*n}^r$ admits two $K_q^r$-decompositions $\mc{Q}_1,\mc{Q}_2$ that do not share a clique; then, if $S\in \mc{Q}_1$, we have that $\mc{Q}_1\setminus S$ and $\mc{Q}_2$ are as desired. 

Let $M$ be a $(q-r)\times q$ Cauchy matrix over $\mathbb{F}_n$ (specifically, we can take $x_i := i$ for $i\in [q-r]$, $y_j := q-r+j$ for $j\in [q]$ and $M_{ij} := 1/(x_i-y_j)$ for $i\in [q-r], j\in [q]$). By construction, every submatrix of a Cauchy matrix is also a Cauchy matrix; a result of Cauchy~\cite{C1841} implies that every square Cauchy matrix is invertible. A vector $v\in \mathbb{F}_n^q$ naturally corresponds to a copy of $K_q^r$ in $K_{q*n}^r$. For $a \in \mathbb{F}_n^{q-r}$, the set of solutions $v\in \mathbb{F}_n^q$ of $Mv=a$ naturally corresponds to a $K_q^r$-decomposition $\mc{Q}_a$ of $K_{q*n}^r$ (since every $r$-set will be in exactly one such $K_q^r$, or equivalently exactly one such solution vector $v$, as every $(q-r)\times (q-r)$ submatrix of $M$ is invertible). Indeed, the family of decompositions $(\mc{Q}_a: a \in \mathbb{F}_n^{q-r})$ partition the $K_q^r$'s in $K_{q*n}^r$ into $K_q^r$-decompositions of $K_{q*n}^r$, since for every vector $v$, there is exactly one $a$ such that $Mv=a$. Since there are at least two choices $a_1,a_2$ of $a$ as $n > 1$, there exist two decompositions $\mc{Q}_{a_1}, \mc{Q}_{a_2}$ of $K_{q*n}^r$ that do not share a clique.
\end{proof}

The following is the partite analog of Lemma~2.4 in~\cite{DKP25}.

\begin{lem}\label{lem:OrthBooster}
Let $q > r\ge 1$ be integers. If $S$ is a $K^r_q$-clique, then there exists an orthogonal $K_q^r$-booster $B$ for $S$ such that $B\cup S$ is $q$-partite.   
\end{lem}

\begin{proof}
Let $B$ be a $q$-partite $K_q^r$-booster for $S$ such that $i$, the number of cliques of $\on(B)$ that contain at least one edge of $S$, is maximized. Note such exists by Lemma~\ref{lem:Booster}; furthermore we have that $i\ge 2$ by the definition of a booster. If $i=\binom{q}{r}$, then $B$ is orthogonal as desired. 

So we assume $i < \binom{q}{r}$. Hence there exists a clique $Q\in \on(B)$ such that $Q$ contains at least two edges of $S$. Let $j:= |V(Q)\cap V(S)|$. Note $j\ge r+1$ since $Q$ contains at least two edges of $S$. Since $Q\ne S$, it follows that $j < q$. Let $B^*$ be a $q$-partite $K_q^r$-booster for $Q$ as given by Lemma~\ref{lem:Booster}.  By relabelling vertices, we assume without loss of generality that $V(B^*) \cap V(B) = V(Q)$ and in addition that $V(Q) \cap V(S)$ is not contained in a clique of $\on(B^*)$, as we now proceed to show is possible.

Since $\on(B^*)$ has at least two cliques that contain at least one edge of $Q$, there exists some subset $T_0\subseteq V(Q)$ with $|T_0|=r+1$ such that $T_0$ is not contained in a clique of $\on(B^*)$. Let $T\subseteq V(Q)$ with $T_0\subseteq T$ and $|T|=j$. Note then that $T$ is not contained in a clique of $\on(B^*)$ and hence letting $V(Q)\cap V(S)=T$ is as desired.

Let $B'=B\cup B^*$. Let $\on(B') := (\on(B)\setminus \{Q\})\cup\on(B^*)$ and $\off(B'):= \off(B)\cup \off(B^*)$. Note that $\on(B')$ is a $K_q^r$-decomposition of $B'\cup S$ and $\off(B')$ is a $K_q^r$-decomposition of $B'$. Also, by Remark~\ref{rem:partite}, $B'$ is $q$-partite. Moreover, the number of cliques of $\on(B')$ intersecting $S$ is at least $i+1$, and hence $B'$ contradicts the maximality of $B$.
\end{proof}

We also recall the following object from~\cite{DKP25}.

\begin{definition}[Hinge]
Let $S$ and $S'$ be two $K_q^r$-cliques such that $S\cap S'=\{e\}$ for some $r$-edge $e$. A \emph{$K_q^r$-hinge} for $S$ and $S'$ is an $r$-graph $H$ such that $H$ is edge-disjoint from $S\cup S'$, $H\cup (S\setminus e)$ has a $K_q^r$-decomposition $\hleft(H)$ (the ``left" one), and $H\cup (S'\setminus e)$ has a $K_q^r$-decomposition $\hright(H)$ (the ``right" one). 

Furthermore, we say $H$ is \emph{independent} if $V(S)\cup V(S')$ is independent in $H$.
\end{definition}

We note that if $B$ is an orthogonal, $q$-partite $K_q^r$-booster for $S$ and $e\in S$ and we let $S'=\on(B)[e]$, then $H:=B\setminus (S'\setminus e)$ is a $K_q^r$-hinge for $S$ and $S'$ where $\hleft(H):= \on(B)\setminus \{S'\}$, $\hright(H):= \off(B)$ and $S \cup S' \cup H$ is $q$-partite. 

The following is the partite analog of Lemma~2.5 in~\cite{DKP25}.

\begin{lem}\label{lem:IndHinge}
Let $q > r\ge 1$ be integers. If $S_1$ and $S_2$ are $K^r_q$-cliques such that $S_1 \cap S_2 = \{e\}$ for some $r$-edge $e$, then there exists an independent $K_q^r$-hinge $H$ for $S_1$ and $S_2$ such that $S_1\cup S_2 \cup H$ is $q$-partite.    
\end{lem}
\begin{proof} 
Let $S$ be a $K^r_q$-clique where $S_1 \cap S = S_2 \cap S = \{e\}$.
As mentioned above it follows from Lemma~\ref{lem:OrthBooster} that there exists a $K_q^r$-hinge $H_1$ for $S_1$ and $S$ where $S \cup S_1 \cup H_1$ is $q$-partite and a $K_q^r$-hinge $H_2$ for $S$ and $S_2$ where $S \cup S_2 \cup H_2$ is $q$-partite. By relabeling vertices, we assume without loss of generality that $V(H_1) \cap V(H_2) = V(S)$. 
Let $H := H_1\cup (S\setminus \{e\}) \cup H_2$.  Let $\hleft(H) := \hleft(H_1)\cup \hleft(H_2)$ and $\hright(H) := \hright(H_1)\cup \hright(H_2)$. Then $H$ is an independent $K_q^r$-hinge for $S_1$ and $S_2$ where $S_1 \cup S_2 \cup H$ is $q$-partite, as desired. 
\end{proof}

\subsection{Counting Edge-Intersecting Absorbers}\label{subsec:EdgeIntAbsorbers}

Now we are prepared to prove~\cref{thm:PartiteAbsorberExistence}, the existence of a rooted $q$-partite edge-intersecting $K_q^r$-absorber.
Our construction follows that in~\cite{DKP25}, where we additionally verify that it is edge-intersecting and further that it satisfies the definition of rooted partite degenerate in~\cref{def:PartiteDegeneracy}. 

\begin{lateproof}{thm:PartiteAbsorberExistence}
We assume without loss of generality that $v(L)\ge q+r$ since it suffices to prove the theorem for this case (to see this, if $v(L) < q+r$, we let $L'$ be obtained from $L$ by adding isolated vertices such that $v(L')=q+r$; then we find a $K_q^r$-absorber $A$ for $L'$ which is then also a $K_q^r$-absorber for $L$). 

By Theorem~\ref{thm:EdgeIntersectingIntegral}, there exists an $L$-edge-intersecting integral $K_q^r$-decomposition $\Phi$ of $L$ (with integers $w_Q$ for each $Q\in \binom {V(L)}{q}$). We view $\Phi$ as a multi-set of positive cliques $\Phi^+$ and a multi-set of negative cliques $\Phi^-$ (where each clique $Q$ appears with multiplicity $|w_Q|$). 

For each edge $e\in L$, let $M_e$ be a directed matching from the elements of $\Phi^-$ containing $e$ to all but one of the elements of $\Phi^+$ containing $e$. For each $r$-set $f$ in $\binom{V(L)}{r}\setminus L$, let $M_f$ be a directed matching from the elements from $\Phi^-$ containing $f$ to the elements of $\Phi^+$ containing $f$.

We construct a $q$-partite $L$-edge-intersecting $K_q^r$-absorber $A$ for $L$ as follows:
\begin{equation*}
    A := \left.\bigcup_{S \in \Phi^+ \cup \Phi^-}B_S\right. \cup \left.\bigcup_{f \in \binom{V(L)}{r},Q_1Q_2\in M_f}H_{f,Q_1,Q_2}\right.\text{\qquad where}
\end{equation*}
\begin{itemize}\itemsep.05in
    \item for each clique $S \in \Phi^+ \cup \Phi^-$, $B_S$ is an orthogonal $K_q^r$-booster for $S$ such that $B_S \cup S$ is $q$-partite whose existence is guaranteed by Lemma~\ref{lem:OrthBooster}, and
    \item for each $f \in \binom{V(L)}{r}$ and each $Q_1Q_2 \in M_f$,  $H_{f,Q_1,Q_2}$ is an independent $K^r_q$-hinge for $\on(B_{Q_1})[f]$ and $\on(B_{Q_2})[f]$ such that $\on(B_{Q_1})[f] \cup \on(B_{Q_2})[f] \cup H_{f,Q_1,Q_2}$ is $q$-partite whose existence is guaranteed by Lemma~\ref{lem:IndHinge} since 
     $\on(B_{Q_1})[f] \cap \on(B_{Q_2})[f] = \{f\}$, as $B_{Q_1}$ and $B_{Q_2}$ are orthogonal boosters. 
\end{itemize}
\vskip.05in

We assume without loss of generality that $V(B_S)\setminus V(S)$ is disjoint from every other booster. Similarly we assume without loss of generality that $V(H_{f,Q_1,Q_2})\setminus (\on(B_{Q_1})[f] \cup \on(B_{Q_2})[f])$ is disjoint from every other booster and hinge.

First note that $A$ is $L$-edge-intersecting because $\Phi$ is an $L$-edge-intersecting integral decomposition, and so every $B \in \left.\bigcup_{S \in \Phi^+ \cup \Phi^-}B_S\right.$ and every $H \in \left.\bigcup_{f \in \binom{V(L)}{r},Q_1Q_2\in M_f}H_{f,Q_1,Q_2}\right.$ is also $L$-edge-intersecting. Likewise, by repeated applications of \cref{rem:partite}, $A$ is $L$-rooted $q$-partite.

We also argue that $A$ is $(2, q)$-rooted partite degenerate. Note that there is a natural partition 
$$V(A)\setminus V(L) = \bigcup_{S \in \Phi^+\cup \Phi^-} \left( V(B_S)\setminus V(S) \right) \cup \bigcup_{f \in \binom{V(L)}{r}, Q_1Q_2 \in M_f} \left(V(H_{f, Q_1, Q_2}) \setminus \left(\on(B_{Q_1})[f] \cup \on(B_{Q_2})[f] \right)\right).$$
Fix an ordering of the elements of $\Phi^+\cup \Phi^-$ and an ordering of the tuples $(f, Q_1, Q_2)$ where $f \in \binom{V(L)}{r}$ and $Q_1Q_2 \in M_f$. 
Since each booster $B_S$ and hinge $H_{f, Q_1, Q_2}$ is $q$-partite by construction, the first condition of Definition~\ref{def:PartiteDegeneracy} is satisfied. Additionally, observe that each booster $B_S$ is $S$-rooted and each hinge $H_{f, Q_1, Q_2}$ is $\on(B_{Q_1})[f] \cup \on(B_{Q_2})[f]$-rooted where each $S, \on(B_{Q_1})[f],$ and $\on(B_{Q_2})[f]$ are isomorphic to $K_q^r$. Thus, each $B_S$ and $H_{f, Q_1, Q_2}$ has at most $2q$ roots and the second condition of Definition \ref{def:PartiteDegeneracy} is satisfied, as desired. 

Now we proceed to verify that $A$ is a $K_q^r$-absorber for $L$ as follows. Recall that by Definition~\ref{def:booster} each $B_S$ is edge-disjoint from $S$ and each hinge in the above construction is an independent hinge. Therefore, we have that $V(L)$ is independent in $A$ and $A$ is simple (that is the boosters and hinges are all pairwise edge-disjoint).

It remains to show that there exist $K_q^r$-decompositions $\mc{A}_1$ of $A$ and $\mc{A}_2$ of $A\cup L$ as follows:
\begin{align*} &\mc{A}_1: = \bigcup_{S\in \Phi^-} \bigg( \mc{O}{\rm n}(B_S)~\Big\backslash~~~~~\bigcup_{e\in S}~~~~~~\mc{O}{\rm n}(B_S)[e] \bigg) ~&\cup~&\bigcup_{f\in \binom{V(L)}{r}, Q_1Q_2\in M_f} \hleft(H_{f,Q_1,Q_2})~&\cup~&\bigcup_{S\in \Phi^+} \mc{O}{\rm ff}(B_S),\\
&\mc{A}_2 := \bigcup_{S\in \Phi^+} \bigg(\on(B_S)~\Big\backslash~\bigcup_{e\in S: S\in V(M_e)}\on(B_S)[e] \bigg)~&\cup~&\bigcup_{f\in \binom{V(L)}{r}, Q_1Q_2\in M_f} \hright(H_{f,Q_1,Q_2})~&\cup~&\bigcup_{S\in \Phi^-} \mc{O}{\rm ff}(B_S).\end{align*}
In other words, for $\mc{A}_1$, we use the \emph{on} decompositions of the negative clique boosters except we do not use the negative orthogonal cliques themselves; we use the \emph{left} decomposition of the hinges to decompose the edges of those unused cliques; finally we use the \emph{off} decompositions of the positive clique boosters.

For $\mc{A}_2$, we use the \emph{on} decompositions of the positive clique boosters except we do not use the positive orthogonal cliques themselves (except for those unmatched ones which decompose the edges of $L$); we use the \emph{right} decomposition of the hinges to decompose the edges of those unused cliques; finally we use the \emph{off} decompositions of the negative clique boosters.
\end{lateproof}

We remark for the observant reader that Theorem~\ref{thm:PartiteAbsorberExistence} implies that, if $L$ is a $q$-partite $r$-graph, then there exists a $q$-partite $L$-edge-intersecting $K_q^r$-absorber $A$ for $L$ such that $A \cup L$ is also $q$-partite. 

We are now ready to prove~\cref{thm:CountingAbsorbers}.

\begin{lateproof}{thm:CountingAbsorbers}
    This follows from~\cref{thm:PartiteAbsorberExistence} and~\cref{lem:EmbedDegenPartite}.
\end{lateproof}

\section{High Minimum Degree Omni-Absorber Theorem}\label{Sec:Omni}

In this section, we prove~\cref{thm:NWRefinedEfficientOmniAbsorber}. As described in~\cref{Subsec:PfOverviewAbsorbers}, we know there exists an omni-absorber $A$ in $K_n^r$ for our reserves graph $X$ by Theorem \ref{thm:RefinedEfficientOmniAbsorber}, but $A$ may use edges from $K_n^r$ that are not present in $G$.
Thus, we generalize the strategy employed by Delcourt, Lesgourgues, and the authors in~\cite{EMeetsNW} and described by the second author in~\cite{P25Survey} to the hypergraph setting: first, for each edge of $A$ we embed into $G$ a fake-edge, a gadget developed by Delcourt and the second author in~\cite{DPI} with the same divisibility properties as an edge. Then, for every $F\cong K^r_q$ in the decomposition family of $A$, we embed a private absorber into $G$ for the graph $F'$, built from $F$ by replacing the edges of $A$ by the embedding of their respective fake-edges. By the General Embedding Lemma presented in~\cite{EMeetsNW}, it suffices to show that each fake-edge and private absorber has many embeddings into $G$. This is accomplished in~\cite{EMeetsNW} through the low-rooted degeneracy of each gadget. Since the rooted degeneracy of hypergraph fake-edges remains low enough, we directly apply this strategy to embed fake edges. However, we overcome the insufficient rooted degeneracy of hypergraph absorbers by using~\cref{thm:PartiteAbsorberExistence}. In~\cref{subsec:GenEmbed}, we state the General Embedding Lemma from~\cite{EMeetsNW} and its required definitions, before proving~\cref{thm:NWRefinedEfficientOmniAbsorber} in~\cref{subsec:FakeEdges}.

\subsection{A General Embedding Lemma}\label{subsec:GenEmbed}

As mentioned, we require the General Embbedding Lemma from~\cite{EMeetsNW} but in order to state it we need the following definitions. First, recall that a hypergraph $W$ is a \emph{supergraph} of hypergraph $J$ if $V(J)\subseteq V(W)$ and $E(J)\subseteq E(W)$. Next, we require the definition of a supergraph system to describe the relationship between the gadgets we wish to embed and the subgraphs we wish to embed them on.

\begin{definition}[Edge-Intersecting, $C$-bounded Supergraph System]
Let $\mc{H}$ be a family of subgraphs of a hypergraph $J$. A \emph{supergraph system} $\mc{W}$ for $\mc{H}$ is a family $(W_H : H\in \mc{H})$ where for each $H\in \mc{H}$, $W_H$ is a supergraph of $H$ with $(V(W_H)\setminus V(H))\cap V(J)=\emptyset$ and for all $H'\ne H\in \mc{H}$, we have $V(W_H)\cap V(W_{H'}) \setminus V(J)= \emptyset$. We let $\bigcup \mc{W}$ denote $\bigcup_{H\in \mc{H}} W_H$ for brevity. We say $\mc{W}$ is \emph{edge-intersecting} if $W_H$ is $J$-edge-intersection for every $H\in \mc{H}$. For a real $C \geq 1$, we say that $\mc{W}$ is \emph{$C$-bounded} if $\max\{e(W_H),~v(W_H)\}\le C$ for all $H\in \mc{H}$.
\end{definition}

Next we require the notion of embedding a supergraph system into a host graph $G$ as follows.

\begin{definition}[Embedding a Supergraph System]
Let $J$ be a hypergraph and let $\mc{H}$ be a family of subgraphs of $J$. Let $\mc{W}$ be a supergraph system for $\mc{H}$. Let $G$ be a supergraph of $J$. An \emph{embedding} of $\mc{W}$ \emph{into} $G$ is a map $\phi : V(\bigcup \mc{W}) \hookrightarrow V(G)$ preserving edges such that $\phi(v)=v$ for all $v\in V(J)$. We let $\phi(\mc{W})$ denote $\bigcup_{e\in \bigcup W} \phi(e)$ (i.e.~the subgraph of $G$ corresponding to $\bigcup \mc{W}$).
\end{definition}

We are now prepared to state the General Embedding Lemma. The statement restricted to graphs appears in~\cite{EMeetsNW}, but the hypergraph version was provided in the appendix as Lemma~B.3. For the interested reader, we note that the lemma (and the above definitions) easily generalize to the multi-graph setting; we further note that the lemma easily generalizes to the setting where each $W_H$ is instead a family $\mc{W}_H$ of possible $W_H$ (namely by simply applying the lemma with $\gamma':= \frac{\gamma}{C\cdot 2^{C^r}}$ say and where we choose $W_H$ as the most common member of $\mc{W}_H$) but we do not require such generalizations here.

\begin{lem}[General Embedding]\label{lem:EmbedGeneral}
For every $C > r \ge 1,$ and$~\gamma\in (0,1]$, there exists $C'\ge 1$ such that the following holds for large enough $n$. Let $J\subseteq G\subseteq K_n^r$ with $\Delta(J)\le \frac{n}{C'}$. Suppose that $\mc{H}$ is a $C$-refined family of subgraphs of $J$ and $\mc{W}$ is a $C$-bounded edge-intersecting supergraph system of $\mc{H}$. If for each $H\in \mc{H}$ there exist at least $\gamma\cdot n^{|V(W_H)\setminus V(H)|}$ embeddings of $W_H$ into $G\setminus (E(J)\setminus E(H))$,
then there exists an embedding $\phi$ of $\mc{W}$ into $G$ such that $\Delta(\phi(\mc{W})) \le C'\cdot \Delta(J)$.
\end{lem} 

In~\cite{EMeetsNW}, a proof of~\cref{lem:EmbedGeneral} appears using an ``avoid the bad'' technique. For the interested reader, we include an alternate proof in the appendix using a ``slotting'' technique developed by Delcourt, Kelly, and the second author in~\cite{RefIII}.

\subsection{Embedding Fake Edges and Private Absorbers}\label{subsec:FakeEdges}

Next, we recall the following gadgets from~\cite{DPI}.

\begin{definition}[Anti-Edge and Fake-Edge]\label{def:AntiEdge}
Let $q>r\ge 1$. Let $f$ be a set of vertices of size $r$.\vspace*{0.05in}
\begin{itemize}\itemsep.05in
    \item An \emph{anti-edge on $f$}, denoted ${\rm AntiEdge}_q^r(f)$, is a set of new vertices $x_1,\ldots, x_{q-r}$ together with edges $\binom{f\cup \{x_i:i\in[q-r]\} }{r} \setminus \{f\}$.
    \item A \emph{fake-edge on $f$}, denoted ${\rm FakeEdge}_q^r(f)$, is a set of new vertices $x_1,\ldots, x_{q-r}$ together with a set of anti-edges $\{ {\rm AntiEdge}_q^r(T): T\in \binom{f\cup \{x_i:i\in[q-r]\} }{r} \setminus \{f\} \}$.
\end{itemize}
\end{definition}

Observe that fake edges have the same divisibility properties as actual edges, that is, if $F={\rm FakeEdge}_q(S)$, then for every $0 \le i \le r-1$ and $S' \subseteq S$ with $|S'| = i$, we have $d_F(S') \equiv 1 \mod \binom{q-i}{r-i}$. Therefore, replacing an edge $e$ of a graph $J$, where $V(e)=S$, by a fake edge on $S$ maintains the divisibility properties of $J$. The following properties of fake-edges will also be useful.

\begin{proposition}\label{prop:FakeEdgeProperties}
    Let $q>r\ge 1$. Let $f$ be a set of vertices of size $r$.\vspace*{0.05in}
    \begin{itemize}\itemsep.05in
        \item ${\rm FakeEdge}_q^r(f)$ has at most $(q-r)\cdot \binom{q}{r} \le q^{r+1}$ vertices and at most $q^{r(r+1)}$ edges.
        \item ${\rm FakeEdge}_q^r(f)$ has degeneracy rooted at $f$ at most $\binom{q-1}{r-1}$.
    \end{itemize}
\end{proposition}
We note for the interested reader that a collection of (well-chosen) fake edges can be viewed as a \textit{refiner}, a key concept in the proof of~\cref{thm:RefinedEfficientOmniAbsorber} from \cite{DPI}, which we do not require here.

The following lemma concerns embedding hypergraphs with small rooted degeneracy into a high minimum degree host hypergraph. This is the hypergraph analog of Lemma~3.6 in~\cite{EMeetsNW}.

\begin{lem}\label{lemma:DegeneracyEmbedding}
    For all $\varepsilon\in (0,1)$ and $C, d\ge 1$, there exists $\gamma \in (0,1)$ such that for all large enough $n$ the following holds: Let $G\subseteq K_n^r$ with $\delta(G) \ge (1-\frac{1}{d}+\varepsilon)n$. Let $H$ be an $r$-graph and $R\subseteq V(H)$ such that $v(H)\le C$, $R$ is independent in $H$, and $H$ has rooted degeneracy at $R$ at most $d$. If $R'\subseteq V(G)$ with $|R|=|R'|$, then there exist at least $\gamma\cdot n^{|V(H)\setminus R|}$ embeddings $\phi$ of $H$ into $G$ with $\phi(R)=R'$. 
\end{lem}

\begin{proof}
    Let $h=|V(H)|$ and $w=|V(W)|$. Let $\gamma := (1-\frac{1}{d} + \frac{\eps}{2})^{d\cdot |V(H)\setminus R|}.$ Fix an ordering $v_1,\ldots, v_{w-h}$ of the vertices $V(W)\setminus V(H)$ such that for all $i\in [w-h]$, we have $|N_W(v_i)\cap (V(H)\cup \{v_j: 1\le j < i\})| \leq d$. We embed the vertices $v_1,\ldots, v_{w-h}$ one at a time in $G$. For every $i\in [w-h]$, there are at least $(1-\frac{1}{d} + \eps)^d\cdot n-C$ choices for where to embed the vertex $v_i$. Thus, since $n$ is large enough, there are at least $((1 - \frac{1}{d} + \eps)^dn-C)^{|V(H)\setminus R|} \ge \gamma \cdot n^{|V(H)\setminus R|}$ embeddings of $W$ in $G$.
\end{proof}

We thus have the required tools to embed a refined omni-absorber into a high minimum degree graph $G$. We are now prepared to prove~\cref{thm:NWRefinedEfficientOmniAbsorber} as follows.

\begin{proof}[Proof of Theorem~\ref{thm:NWRefinedEfficientOmniAbsorber}]
Fix integers $q > r$. Let $C_0$ be an integer such that~\cref{thm:RefinedEfficientOmniAbsorber} (the existence of $C_0$-refined $K^r_q$-omni-absorbers in $K^r_n$) holds. Let $c_1 >0$ and $\gamma_1\in(0,1)$ be such that~\cref{thm:CountingAbsorbers} holds for $t = \binom{q}{r}\cdot q^{r+1}$. Recall from the proof of~\cref{thm:CountingAbsorbers} that $c_1< 1$ and there exists an integer $C_A \ge 1$ depending only on $q, r,$ and $t$ such that the absorbers constructed in the proof of~\cref{thm:CountingAbsorbers} have at most $C_A$ vertices. Thus there exists $\eps > 0$ such that $c_1/\binom{q}{r-1} \le 1/\binom{q-1}{r-1} - \eps.$ Let $\gamma_2\in(0,1)$ be such that~\cref{lemma:DegeneracyEmbedding} holds for $\eps$, $q^{r+1}$, and $d=\binom{q-1}{r-1}$. Let $\gamma:= \min\{\gamma_1, \gamma_2\}$ and $C_E := \max\{q^{r(r+1)},\ q^r + \binom{C_0}{r}\}.$ Let $C'\geq 1$ such that the General Embedding Lemma,~\cref{lem:EmbedGeneral}, holds for $C_E$ and $\gamma$. Finally, let $c:= c_1/2$ and $C:= \binom{q}{r-1} \left(1 + 2 \cdot (C')^2 \cdot C_0 \right)/c$.

Let $n$ be as large as necessary throughout the following argument. Let $X\subseteq  G\subseteq K^r_n$ with $\delta(G)\ge \left(1 - \frac{c}{\binom{q}{r-1}} \right) n$ and $\Delta(X) \leq n/C$. Let $\Delta:=\max\left\{\Delta(X),~n^{1- 1/r}\cdot \log n\right\}$. By~\cref{thm:RefinedEfficientOmniAbsorber}, there exists a $C_0$-refined $K^r_q$-omni-absorber $A_0$ for $X$ in $K_n^r$ with $\Delta(A_0) \leq C_0\cdot \Delta\leq C_0 \cdot n/C$, with decomposition family $\mc{F}_0$ and decomposition function $\mc{Q}_0$. Recall that we aim to prove that there exists a $C$-refined $K_q$-omni-absorber $A$ for $X$ in $G$ with $\Delta(A) \le C\cdot \Delta$. To accomplish this goal, we use the General Embedding Lemma to replace each edge of $A_0$ by a fake-edge in $G$, and therefore each copy of $K^r_q$ in its decomposition family by a hypergraph $F$, and then find a private absorber in $G$ for each such hypergraph~$F$. In order to apply the General Embedding Lemma, we require that $A_0$ be a subgraph of our host graph, though we do not embed upon the edges of $A_0$ at any point in this construction. To that end, let $G_0 := (G\cup A_0)\setminus X$.

First we replace the edges of $A_0$ in $G_0$ with fake-edges in $G\setminus X$. Let $J:=A_0$ and $\mc{H}:=\set*{\{e\}\colon e\in J}$. We define a supergraph system $\mc{W}:=(W_e: \{e\}\in \mc{H})$ where $W_e:= {\rm FakeEdge}^r_q(e)$. We start by embedding $\mc{W}$ into $G_0$. Observe that, by the definition of $C$, we have $(1+C_0)/C\leq c /\binom{q}{r-1}$. Therefore, since $n$ is large enough and $\eps$ satisfies $2c/\binom{q}{r-1} \le 1/\binom{q-1}{r-1} - \eps$, for every edge $e\in J$ we have:

\[ \delta(G_0\setminus(J\setminus \{e\})) \ge \delta(G) - \Delta(X) - \Delta(A_0) \ge \left(1 - \frac{c}{\binom{q}{r-1}} - \frac{1+C_0}{C} \right)n \ge \left(1 - \frac{2c}{\binom{q}{r-1}} \right)n \ge \left(1 - \frac{1}{\binom{q-1}{r-1}} + \eps \right)n.\]

Observe that, by~\cref{prop:FakeEdgeProperties}, each $W_e$ has at most $(q-r) \cdot \binom{q}{r} \le q^{r+1}$ vertices and degeneracy rooted at $V(e)$ at most $\binom{q-1}{r-1}$. Therefore, by~\cref{lemma:DegeneracyEmbedding}, there are at least $\gamma\cdot n^{|V(W_e)\setminus V(e)|}$ embeddings of $W_e$ into $G_0\setminus (J\setminus \{e\})$. Note that $\mathcal{H}$ is an $r$-refined family and, since each $W_e$ has at most $q^{r(r+1)}$ edges, $\mc{W}$ is a $C_E$-bounded, edge-intersecting family. Since $C_0 \le C / C'$ guarantees $\Delta(J) \le n/C'$, by~\cref{lem:EmbedGeneral}, there exists an embedding $\phi$ of $\mc{W}$ into $G_0$ such that 
$$\Delta(\phi(\mc{W})) \le C' \cdot \Delta(A_0)\le C' \cdot C_0 \cdot \frac{n}{C}.$$
In particular, note that $\phi(\mc{W}) \subseteq G \setminus X$. 
\bigskip

 Let $J':=X\cup \phi(\mc{W})$. Observe that since $C\geq 2(C')^2\cdot C_0$, we have 
    $$\Delta(J')\leq \Delta(X)+ C'\cdot \Delta(A_0)\leq 2\cdot C'\cdot C_0\cdot \frac{n}{C}\leq \frac{n}{C'}.$$  
      Let
    \[\mc{H}':=\big\{ (F\cap X)\cup \bigcup_{e\in F\setminus X} \phi(W_e)\colon F\in \mc{F}_0\big\}.\] 
Observe that $\mc{H}'$ is a $C_0$-refined family of subgraphs of $J'$. Also note that by our choice of $C$, we have $(1+C'\cdot C_0)/C \le c/\binom{q}{r-1}.$ Thus, for each $H \in \mc{H}'$, the minimum degree of $(G\setminus X) \setminus (J' \setminus H)$ is at least
$$\delta(G) - \Delta(X) - \Delta(\phi(\mc{W})) \ge \left(1 - \frac{c}{\binom{q}{r-1}} - \frac{1+C' \cdot C_0}{C} \right)n \ge \left(1 - \frac{2c}{\binom{q}{r-1}} \right)n = \left(1 - \frac{c_1}{\binom{q}{r-1}} \right)n.$$
Also note that each $H\in\mc{H}'$ is $K_q^r$-divisible since a fake edge has the same divisibility properties as an edge. Furthermore, each $H$ satisfies $v(H) \le \binom{q}{r}q^{r+1}$. Therefore, for each $H\in\mc{H}'$, let $A_{H}$ be an $H$-edge-intersecting $K^r_q$-absorber for $H$ such that there are at least $\gamma \cdot n^{|V(A_{H})\setminus V(H)|}$ embeddings of $A_{H}$ into $(G\setminus X) \setminus (J' \setminus H')$ as guaranteed by~\cref{thm:CountingAbsorbers}, and such that $A_H$ and $A_{H'}$ are edge-disjoint for distinct $H,H'\in\mc{H}'$. We define a supergraph system $\mc{W}':=(W_H: H\in \mc{H}')$ where $W_H=H\cup A_H$. By construction, $\mc{W}'$ is an edge-intersecting family. Also note that every element of $\mc{W}'$ has at most $q^r + \binom{C_0}{r}$ edges, hence $\mc{W}'$ is $C_E$-bounded. Furthermore, since $e(H) \le \binom{q}{r}q^{r(r+1)}$ for every $H \in \mc{H}'$, we have that $\mc{H}'$ is $\binom{q}{r}q^{r(r+1)}$-refined. Therefore, by~\cref{lem:EmbedGeneral}, there exists an embedding $\phi'$ of $\mc{W}'$ into $G_0$ such that $\Delta(\phi'(\mc{W}')) \le C'\cdot \Delta(J')$.

For each $F \in \mc{F}_0$, let $F'\in\mc{H}'$ such that $F':= (F \cap X) \cup \bigcup_{e \in F \setminus X} \phi(W_e)$. Let $A=\phi(\mc{W})\cup\phi'(\mc{W}')$ for each $F'$. We first define a decomposition family $\mc{F}_A$ for $A$ as follows. We include in $\mc{F}_A$ both the $K^r_q$-decomposition of $\phi'(A_{F'})$ and of $F' \cup \phi'(A_{F'})$. Observe that these decompositions exist because $A_{F'}$ is a $K^r_q$-absorber of $F'$.

We then define a decomposition function $\cQ_A$ for $A$ as follows. Given a $K^r_q$-divisible $L \subseteq X$, for every $F\in \mc{F}_0$, we let $\cQ_{A}(L)$ contain the $K^r_q$-decomposition of $F' \cup \phi'(A_{F'})$ if $F \in \cQ_{A_0}(L)$ and of $\phi'(A_{F'})$ if $F \in \mc{F}_0\setminus \cQ_{A_0}(L)$. By construction, we have $\Delta(A)\leq \Delta(\phi(\mc{W}))+\Delta(\phi'(\mc{W}')) \le 2C'\cdot C_0 \cdot \Delta\leq C\cdot \Delta$, while every edge of $A\cup X$ is in at most $C_0+2$ elements in $\mc{F}_A$. We obtain that $A$ is a $C$-refined $K^r_q$-omni-absorber for $X$ in $G$ with decomposition family $\mc{F}_A$ and decomposition function $\cQ_A$, as desired.
\end{proof}

\section{Regularity Boosting}\label{Sec:boosting}

In this section, we prove~\cref{lem:LowWeightFracDecomposition} and, as a corollary, Theorem~\ref{thm:NWRegBoost}. To achieve this, we generalize the strategy from Delcourt, Lesgourgues, and the authors in~\cite{EMeetsNW} to the hypergraph setting: we first use the Fixed Fractional Decomposition Lemma from~\cite{EMeetsNW} and Lang's Inheritance Lemma for Minimum Degree from~\cite{lang2023tiling} to prove~\cref{lem:LowWeightFracDecomposition} by showing that, if $G$ has minimum degree above the fractional decomposition threshold, every edge is contained in many subgraphs of $G$ that admit a fractional decomposition, whose weight we appropriately rescale. Then~\cref{thm:NWRegBoost} follows by random sampling and Chernoff bounds. 

We require the following additional notation and definitions. Given an $F$-weighting $\psi$ of a hypergraph $G$ and an edge $e$ of $G$, we let $\partial\psi(e)$ denote the total weight of $\psi$ over the edge $e$, that is $\partial\psi(e):= \sum_{F: e\subseteq F} \psi(F)$. Throughout this section, if $\cH$ is a family of copies of $K^r_q$ in a graph $G$, we use notation as if $\cH$ is a hypergraph with $V(\cH) = V(G)$ and $E(\cH) = \cH \subseteq \binom{V(G)}{q}$.

The following is a simplified version of Lemma~4.3 in~\cite{EMeetsNW}.

\begin{lem}[Fixed Fractional Decomposition \cite{EMeetsNW}]\label{lem:Fixed}
    For every $r$-graph $F$ and real $\varepsilon$, there exist an integer $n_0\geq 1$, such that following holds for every $n\geq n_0$: Let $G$ be an $r$-graph on $n$ vertices with minimum degree at least $(\delta^*_F+\varepsilon)n$, and let $\varphi:E(G)\to[1-\frac{1}{e(G)},1]$. Then there exists a fractional $F$-packing $\Phi$ of $G$ such that $\partial\Phi(e)=\varphi(e)$ for every edge $e\in G$.
\end{lem}
In \cite{EMeetsNW},~\cref{lem:Fixed} is stated only for graphs, but its proof holds in our hypergraph setting. For completeness, we include said proof in the appendix.

Another ingredient in our proof of~\cref{lem:LowWeightFracDecomposition} is the following special case of Lang's Inheritance Lemma~\cite{lang2023tiling}.

\begin{lem}[Inheritance Lemma for minimum degree~\cite{lang2023tiling}]\label{lem:LangInheritance}
    For every $\varepsilon\in(0,1)$, there exists $s_0$ such that for all $s\geq s_0$, the following holds for every $n$ large enough. Let $\delta\geq 0$, and let $G$ be an $n$-vertex $r$-graph with $\delta(G)\geq (\delta+\varepsilon)(n-1)$. Then for every $r$-set $R\subseteq V(G)$, there are at least $(1-e^{-\sqrt{s}})\binom{n-r}{s-r}$ distinct $s$-sets $S\subseteq V(G)$ such that $R\subseteq S$ and $\delta(G[R])\geq (\delta+\varepsilon/2)(s-1)$.
\end{lem}

Now we are prepared to prove~\cref{lem:LowWeightFracDecomposition}

\begin{proof}[Proof of~\cref{lem:LowWeightFracDecomposition}]
    Observe that we may assume that $\delta^*_F < 1$ as otherwise the statement is vacuously true. Let $s > \frac{2 \delta}{\eps}+1$ such that~\cref{lem:LangInheritance} holds for $\varepsilon$. 
    Let $C:=2\cdot\binom{s-r}{v(F)-r}$. Let $n$ be large enough, and let $G$ be an $n$-vertex graph with $\delta(G)\geq (\delta^*_F+\varepsilon)n$. Let $\cS$ be the family of all $s$-sets in $V(G)$ with minimum degree at least $(\delta^*_F+\varepsilon/2)(s-1)$. For every edge $e\in G$, we define $\varphi(e)$ to be
    \[\varphi(e):=(1-e^{-\sqrt{s}})\cdot \frac{\binom{n-r}{s-r}}{|\{S\in\mc{S}:e\in S\}|}.\] By~\cref{lem:LangInheritance}, every edge $e\in G$ is in at least $(1-e^{-\sqrt{s}})\cdot\binom{n-r}{s-r}$ sets in $\mc{S}$. Therefore, by building, for every $S\in\mc{S}$, an $F$-decomposition $\Phi_S$ of $G[S]$ such that $\partial\Phi_S(e)=\varphi(e)$, and taking the sum of all these decompositions, we obtain an assignment such that the weight on each edge is constant; we then only need a scaling factor to obtain the desired $F$-decomposition. 
    
    Observe that, by~\cref{lem:LangInheritance}, we have $\varphi(e)\in[1-e^{-\sqrt{s}},1]$ and hence as $s$ is large enough, we have $\varphi(e)\geq 1-\frac{1}{s^r}$. Therefore, by~\cref{lem:Fixed}, for each $S\in\cS$, there exists a fractional $F$-packing $\Phi_S$ of $G[S]$ such that $\partial\Phi_S(e) = \phi(e)$ for all $e \in G[S]$. For each $H \in \binom{G}{F},$ we let $\Phi_S(H)=0$ if $V(H)\setminus S\ne\emptyset$ so as to extend the function to all copies of $F$ in $G$.

    Now we define the following fractional $F$-packing of $G$. For every $H\in\binom{G}{F}$, define
    \[\phi(H):= \frac{1}{(1-e^{-\sqrt{s}})\cdot\binom{n-r}{s-r}}\cdot\sum_{S\in\cS}\Phi_S(H).\]
    Since every $H\in \binom{G}{F}$ is in at most $\binom{n-v(F)}{s-v(F)}$ sets $S\in\cS$, we have
    \[\phi(H)\leq \frac{1}{(1-e^{-\sqrt{s}})\cdot\binom{n-r}{s-r}}\cdot \binom{n-v(F)}{s-v(F)}
    = \frac{\binom{s-r}{v(F)-r}}{(1-e^{-\sqrt{s}})\cdot\binom{n-r}{v(F)-r}}
    \leq \frac{C}{\binom{n-r}{v(F)-r}},    \]
    where we used $(1-e^{-\sqrt{s}})\geq 1/2$ for $s\geq 1$. Finally, by~\cref{lem:LangInheritance}, every edge $e\in G$ is in at least $\frac{1-e^{-\sqrt{s}}}{\varphi(e)}\cdot\binom{n-r}{s-r}$ sets in $\cS$. Therefore we obtain 
    \[\partial\phi(e)=\frac{1}{(1-e^{-\sqrt{s}})\cdot\binom{n-r}{s-r}}\cdot\sum_{S\in\cS}\partial\phi_S(e)=\frac{1}{(1-e^{-\sqrt{s}})\cdot\binom{n-r}{s-r}}\cdot\brackets*{\frac{1-e^{-\sqrt{s}}}{\varphi(e)}\cdot\binom{n-r}{s-r}\cdot \varphi(e)} = 1.
    \]
\noindent
    Hence $\phi$ is the desired $C$-low-weight fractional $F$-decomposition of $G$.
\end{proof}

Now~\cref{lem:LowWeightFracDecomposition} will imply~\cref{thm:NWRegBoost} by random sampling as follows.

\begin{lateproof}{thm:NWRegBoost}
    For $q>r \ge 2$ and $\varepsilon\in(0,1)$, let $C\geq 1$ such that~\cref{lem:LowWeightFracDecomposition} holds for $F=K^r_q$. Let $n$ be a large enough integer, and let $J \subseteq K_n^r$ have minimum degree at least $(\max\{\delta_{K^r_q}^*, 1 - \frac{1}{\binom{q-1}{r-1}} \}+\varepsilon)n$. 
    
    Let $\phi$ be a $C$-low-weight fractional $K^r_q$-decomposition of $J$ guaranteed by~\cref{lem:LowWeightFracDecomposition}, and let $d:=\frac{1}{C}\binom{n-r}{q-r}$. Observe that since $\phi$ is $C$-low-weight, we have that $\phi(H)\cdot d \in [0,1]$ for every $H$. Thus, we define $\cH$ to be a random subcollection of copies of $K^r_q$ in $J$, by including every $H\in\binom{J}{K^r_q}$ independently with probability $\phi(H)\cdot d$. Note that for every $e\in J$, we have $\Expect{|\cH(e)|}=\partial\phi(e)\cdot d = d$.

    By a standard application of 
    the Chernoff bound, we see that
    \[\Prob{\Big||\cH(e)|- d\Big|\geq n^{-(q-r)/3}\cdot C\cdot d} \leq 2e^{-(n^{-(q-r)/3}\cdot C)^2 d/3}. \]
    Note that since $n$ is large enough,
    \[2 \cdot \binom{n-r}{q-r} \cdot e^{-(n^{-(q-r)/3}\cdot C)^2 d/3} \le 2 \cdot \binom{n-r}{q-r} e^{- \frac{C (n-r)^{(q-r)/3}}{3 \cdot (q-r)^{(q-r)}}} \le 1.\]
    Therefore, by the Union Bound, there exists a family $\cH$ of copies of $K_q^r$ in $J$ such that every edge $e \in J$ is in $(\frac{1}{C} \pm n^{-(q-r)/3}) \binom{n-r}{q-r}$ copies of $K_q^r$ in $\cH$, as desired. 
\end{lateproof}

\section{The Finish}\label{sec:Finish}

The previous sections contain the necessary tools to perform the second, third, and fourth steps of our proof outline for \cref{thm:main}. In this section, we first include the remaining statements and proofs necessary for the first step of our proof outline, before combining these into a proof of \cref{thm:main}. 

\subsection{Reserves}

First, we prove~\cref{lem:highmindegreserves} using a standard application of the Chernoff bound~\cite{AS16}. 

\begin{lateproof}{lem:highmindegreserves}
    Let $\gamma > 0$ be small enough, in particular such that $\gamma \cdot \binom{q-1}{r-1}<1$ and $\gamma \le \frac{\eps^{q-r}}{2^{q-r}}.$ Let $n$ be a large enough integer. Let $X$ be the spanning subgraph of $G$ obtained by choosing each edge of $G$ independently with probability $p\in [n^{-\gamma},1)$.

    First, for $S \subseteq \binom{v(G)}{r-1}$, we have \(\Expect{|X(S)|}=p\cdot |G(S)| \le p \cdot n.\) Let $A_S$ be the event that $|X(S)| > 2pn$.  By the Chernoff bound and the fact that $n$ is large enough, we find that 
    \[\Prob{A_S} \leq e^{-p\cdot n / 3} < \frac{1}{2 \cdot n^{r-1}}.\]
    Hence, by the Union Bound, 
    \[\Prob{\bigcup_{S \in \binom{v(G)}{r-1}}A_S} < \frac{1}{2}.\]

    For $T \subseteq V(G)$ with $1\le |T|\le q-1$, we define $N_G(T)$ to be the common neighborhood of all $(r-1)$-subsets of $T$ in $G$, that is, 
\[N_G(T):= \left\{ v\in V(G)\setminus T: \{v\} \cup R \in E(G) \text{ for every } R\in \binom{T}{r-1}\right\}.\]
Similarly,
\[N_X(T):= \left\{ v\in V(G)\setminus T: \{v\} \cup R \in E(X) \text{ for every } R\in \binom{T}{r-1}\right\}.\]
Since $\delta(G)\ge (1-\frac{1}{\binom{q-1}{r-1}}+\varepsilon)n$ and $\binom{|T|}{r-1} \le \binom{q-1}{r-1}$, it follows that $|N_G(T)| \ge \eps n$.
Also, for every $v\in N_G(T)$, we have that $\Prob{v\in N_X(T)} = p^{\binom{|T|}{r-1}}$. Thus by Linearity of Expectation, we obtain
\[\Expect{|N_X(T)|} = \sum_{v\in N_G(T)} \Prob{v\in N_G(T)} = p^{\binom{|T|}{r-1}}\cdot |N_G(T)|.\]

Note that $|N_X(T)|$ is the sum of independent Bernoulli $\{0,1\}$-random variables. Let $B_T$ be the event that $|N_X(T)| < p^{\binom{|T|}{r-1}}\cdot |N_G(T)|/2$. By the Chernoff bound, we find that
\[\Prob{B_T} \le \exp \left[-p^{\binom{|T|}{r-1}}\cdot |N_G(T)| / 8\right] < \frac{1}{2\cdot q \cdot n^{q-1}},\]
where the last inequality is because $\gamma \cdot \binom{q-1}{r-1}<1$ and $n$ is large enough. Hence, by the the Union Bound, 
\[\Prob{ \bigcup_{T} B_T} < \frac{1}{2}.\]

Therefore, with positive probability, $\Delta(X)\leq 2pn$ and none of the $B_T$ happen. Fix such an outcome. For $e\in K^r_n\setminus X$, since none of the events $B_T$ occur, the number of copies of $K^r_q$ in $X\cup \{e\}$ containing $e$ is at least
\[
\frac{1}{(q-r)!}\prod_{i=r}^{q-1} \left(p^{\binom{i}{r-1}} \cdot \eps \cdot n \cdot \frac{1}{2} \right) \ge \frac{\eps^{q-r}}{2^{q-r}} \cdot p^{\binom{q}{r}-1} \cdot \frac{n^{q-r}}{(q-r)!} \ge \gamma \cdot p^{\binom{q}{r}-1} \cdot \binom{n}{q-r},
\]
where we use that $\sum_{i = r-1}^{q-1} \binom{i}{r-1} \le \binom{q-1}{r}$ for the first inequality, and the standard bound $\frac{n^k}{k!}\ge \binom{n}{k}$ and the fact that $\gamma \le \frac{\eps^{q-r}}{2^{q-r}}$ for the second inequality. Thus, there exists a spanning subgraph $X$ that satisfies the requirements of~\cref{lem:highmindegreserves}, as desired. 
\end{lateproof}

\subsection{Proof of Theorem~\ref{thm:main}}
 
Now we are prepared to prove~\cref{thm:main}.

\begin{lateproof}{thm:main} 
Note that if $r = 1$, for any integer $q$, the $K_q^r$-divisibility of $G$ is sufficient to guarantee that $G$ admits a $K_q^r$-decomposition. Thus, let $q > r \ge 2$ be integers and $\eps \in (0, 1)$. Let $n$ be as large as necessary throughout the proof. Fix $\eps_0 >0$ such that $\frac{1}{\binom{q}{r-1}} \le \frac{1}{\binom{q-1}{r-1}}+\eps_0$. Let $\eps':= \min \{\eps, \eps_0 \}$. Let $\gamma \in (0, 1)$ be as in~\cref{lem:highmindegreserves} for $q, r,$ and $\eps'$. Let $C \ge 1$ and $c \ge 0$ be as in~\cref{thm:NWRefinedEfficientOmniAbsorber} for $q, r,$ and $\eps'.$ Let $c_b \in (0, 1)$ be as in~\cref{thm:NWRegBoost} for $q, r,$ and $\eps/2$. Let $\rho \in (0,1)$ be such that $\parens*{c_b- n^{-1/3}}\cdot \binom{n-r}{q-r} \le  \parens*{\rho- n^{-1/3}}\cdot \binom{n}{q-r}$ and $\parens*{c_b+ n^{-1/3}}\cdot \binom{n-r}{q-r} \ge  \parens*{\rho+ n^{-1/3}}\cdot \binom{n}{q-r}$ for large enough $n$, and let $\alpha \in (0, 1)$ be as in~\cref{cor:NibbleReservesSpecific} for $\rho$. Finally, let $\sigma:= \min \left\{\gamma, \frac{\alpha - 1}{ \binom{q}{r} - 1} \right\}$, and let $p:= n^{-\sigma}$ and $\Delta:= 2pn$. 

Let $G$ be an $r$-graph such that $\delta(G) \ge \left( \max\left\{ \delta_{K_q^r}^*+ \eps, 1 -\frac{c}{\binom{q}{r-1}} \right\}\right) n$. Recall from the proof of~\cref{thm:NWRefinedEfficientOmniAbsorber} that $c \le 1$. Hence, $\delta(G) \ge \left(1 - \frac{1}{\binom{q-1}{r-1}} + \eps' \right)n$. Therefore, by~\cref{lem:highmindegreserves}, there exists $X \subseteq G$ with $\Delta(X) \leq \Delta$ such that every $e \in G\setminus X$ is in at least $\gamma \cdot p^{\binom{q}{r}-1}\binom{n}{q-r}$ $K_q^r$-cliques of $X \cup \{e\}$.

Observe that since $n$ is large enough, $p = n^{-\sigma} \le 1/2C$. Therefore, $\Delta(X) \le 2 \cdot n^{-\sigma} \cdot n \le \frac{n}{C}$, and hence, by~\cref{thm:NWRefinedEfficientOmniAbsorber}, there exists a $C$-refined $K^r_q$-omni-absorber $A$ for $X$ with $A \subseteq G$ and $\Delta(A) \le C\cdot \Delta$. 

Let $J:= G\setminus (A\cup X)$. Now we have
\begin{align*}
    \delta(J) & \ge \delta(G) - \Delta(A) - \Delta(X) \\
    & \ge \left( \max\left\{ \delta_{K_q^r}^* +\eps, 1 -\frac{c}{\binom{q}{r-1}} \right\}\right) n - (1 +C)n \\
    & \ge \left(\max \left(\delta_{K_q^r}^*, 1 - \frac{1}{\binom{q-1}{r-1}}\right) +\frac{\eps'}{2} \right)n,
\end{align*}
where the last inequality is because $(1+C)\cdot 2 \cdot n^{-\sigma} \le \frac{\eps'}{2}$ since $n$ is large enough. Thus, by~\cref{thm:NWRegBoost}, there exists a family $\mc{H}$ of copies of $K_q^r$ in $J$ such that every $e\in J$ is in $\parens*{c_b\pm n^{-(q-r)/3}}\cdot \binom{n-r}{q-r}$ copies of $K^r_q$ in $\mc{H}$. 

By our choice $p$,
\[\gamma \cdot p^{\binom{q}{r}-1}\binom{n}{q-r} \ge \frac{\gamma}{(q-r)^{q-r}} \cdot n^{- (\alpha-1) \left(\binom{q}{r}-1\right)/\left(\binom{q}{r}-1\right)}n^{q-r} = \frac{\gamma}{(q-r)^{q-r}}n^{q - r - \alpha +1} \ge n^{q - r - \alpha,}\] 
since $n$ is large enough. Therefore, by our choice of $X$, every edge of $J$ is in at least $K_q^r$-cliques of $X \cup \{e\}$. Also, by our choice of $\rho$, every edge $e\in J$ is in $\parens*{\rho\pm n^{-1/3}}\cdot \binom{n}{q-r}$ copies of $K^r_q$ in $\mc{H}$. Hence, by~\cref{cor:NibbleReservesSpecific}, there exists a $K_q^r$-packing $Q_1$ of $J \cup X$ covering all edges of $J$. Observe that since $Q_1$ is $K_q^r$-divisible since it admits a $K_q^r$-decomposition $\mc{Q}_1$. 

Let $L:=X \setminus Q_1$. Also note that $A$ is $K_q^r$-divisible as it admits a $K_q^r$-decomposition by the definition of omni-absorber. Since $G$ is $K_q^r$-divisible by assumption and $A$ and $Q_1$ are edge-disjoint $K_q^r$-divisible subgraphs of $G$, it follows that $L = G\setminus (Q_1 \cup A)$ is $K_q^r$-divisible. By the definition of omni-absorber, it follows that $L\cup A$ admits a $K_q^r$-decomposition $\mc{Q}_2$. Thus, $\mc{Q}_1 \cup \mc{Q}_2$ is a $K_q^r$-decomposition of $G$, as desired.
\end{lateproof}

\section{Further Directions}\label{Sec:Further}

Despite the recent breakthroughs on designs and hypergraph decompositions, many interesting open questions about hypergraph decompositions remain. As mentioned in the survey by the second author~\cite{P25Survey}, it would be interesting to prove a unified theorem of designs incorporating the various generalizations of Existence. In particular, it would be interesting to prove a hypergraph version of the ``\Erdos meets Nash-Williams' Conjecture" (see~\cite{EMeetsNW} for more history on this conjecture) which would thus generalize the High Girth Existence Conjecture which was proved by Delcourt and the second author in~\cite{DPII} using refined absorption. Similarly, it would be interesting to prove a probabilistic threshold (or spread) version of Nash-Williams' Conjecture and more generally of the hypergraph version of Nash-Williams' Conjecture (we refer the reader to~\cite{DKPIV} for further history of probabilistic thresholds for decompositions). 

The interested reader might wonder whether it is possible to formulate a very general set of conditions guaranteeing the existence of $K_q^r$-decompositions - at least in dense hypergraphs. Some notable examples of such general theorems were obtained for hypergraph matchings by Keevash and Mycroft~\cite{KM15} and for hypergraph tilings by Lang~\cite{lang2023tiling}.  

It is not too hard in light of the theory developed in~\cite{DPI} and~\cite{EMeetsNW} to modify the proofs in this paper to yield the following such theorem. (Of course the following would not capture the intricacies of the high girth, spread or random graph settings which require more direct access to the package of tools from refined absorption). Recall from~\cite{EMeetsNW} that a fractional $K_q^r$-decomposition $\psi$ of an $r$-graph $G$ is \textit{$\alpha$-seeded} if every copy of $K_q^r$ has weight at least $\frac{\alpha}{ \binom{n-r}{q-r}}$, and $\psi$ is \textit{$\eps$-almost} if $\partial\psi(e) \in [1-\varepsilon,~1]$ for each edge $e$ of $G$.

\begin{thm}
$\forall~C \ge 1$, $\exists~\varepsilon>0$ such that the following holds for sufficiently large $n$: If $G\subseteq K_n^r$ is $K_q^r$-divisible and satisfies both of the following
\vskip.05in
\begin{itemize}\itemsep.05in
    \item[(i)] \textbf{Robust Fractional:} $\exists$~a $C$-low-weight $\frac{1}{C}$-seeded $\varepsilon$-almost fractional $K_q^r$-decomposition of $G$, and
    \item[(ii)] \textbf{Many Boosters and Hinges:}  $\forall~ K_q^r$-clique $S$ in $K_n^r$, $\exists$~a $K_q^r$-booster $B$ with at least $\frac{1}{C}\cdot n^{v(B)-v(L)}$ copies of $B$ rooted at $V(S)$ in $G$, and $\forall~ K_q^r$-cliques $S_1,S_2$ in $K_n^r$ with $|V(S_1)\cap V(S_2)|=r$, $\exists$~a $K_q^r$-hinge $H$ with at least $\frac{1}{C}\cdot n^{v(H)-v(S_1\cup S_2)}$ copies of $H$ rooted at $V(S_1)\cup V(S_2)$ in $G$,
\end{itemize}
\vskip.05in
then there exists a $K_q^r$-decomposition of $G$.
\end{thm}

We did not pursue the explicit derivation of such a theorem in this paper - in part, as its main application would be in generating progress towards the hypergraph version of Nash-Williams' Conjecture, as accomplished by our main results. The other reason is that the second author and Delcourt in a forthcoming work prove a more general version of the above theorem that also works for $F$-decompositions for any hypergraph $F$ (thus adapting the method of refined absorption the case of general $F$). Hence another future direction would be to utilize the method of refined absorption and the surrounding theory to provide bounds on the (fractional) decomposition threshold for a general hypergraph $F$ (see Yuster~\cite{Yuster12} and the work of Glock, K\"uhn, Lo, Montgomery and Osthus~\cite{GKLMO19} for such results about graphs). 

\section*{Acknowledgements}
The authors would like to thank Michelle Delcourt and Thomas Lesgourgues for many helpful comments and discussions during the preparation of this manuscript. 

\bibliographystyle{plain}
\bibliography{bibliography}

\appendix

\section{Proof of the Edge-Density Inheritance Lemma}

For completeness, here we include a proof of~\cref{lem:e(G[S])large}, using the method suggested by Keevash in~\cite{K11}.
In the following argument, we use ordered $k$-sets to simplify the calculations related to the second-moment method.

\begin{lateproof}{lem:e(G[S])large}
Let $S = (v_1, \ldots, v_k)$ where for each $i \in [k]$, the vertex $v_i$ is chosen uniformly at random from $V(G)$. First, we bound the probability that $S$ is not a set of $k$ distinct vertices.
To that end, observe that
$$\frac{n^k}{\binom{n}{k}\cdot k!} = \prod_{i=1}^{k-1} \left(1 + \frac{i}{n-i}\right) \leq \exp\left( \sum_{i=1}^{k-1}~\frac{2i}{n}\right) \leq \exp \left(\frac{2k^2}{n}\right)\leq e,$$
where we used that $n \ge 2k^2$ for both inequalities. 
Hence the probability that all the $v_i$ are distinct (equivalently, that~$S$ corresponds to an ordered set) is at least $1/e$.  

For each $T \in \binom{[k]}{r}$, let $X_T$ be the indicator variable that is 1 when $\{v_i: i\in T\} \in E(G)$ and 0 otherwise. Observe that $\Pr[X_T = 1]= \frac{e(G)\cdot r!}{n^r}$. Let $X := \sum_{T \in \binom{[k]}{r}} X_T$. We bound the variance of $X$, ${\rm Var}[X]$, as follows:
\begin{align*}
    {\rm Var}[X] \le |\E[X^2] - \E[X]^2| & \leq \sum_{(S, T) \in \binom{[k]}{r}^2}\Big|\Pr[X_T=X_S = 1] - \Pr[X_T=1]\cdot \Pr[X_S = 1]\Big| \\
     & \leq \sum_{\substack{(S, T) \in \binom{[k]}{r}^2, \\ S \cap T = \emptyset}} 0 + \sum_{\substack{(S, T) \in \binom{[k]}{r}^2, \\ S \cap T \neq \emptyset}} 1 \\
     & \leq~\binom{k}{r}\cdot r \cdot \binom{k-1}{r-1},
\end{align*}
where in the first sum of the second line, the value $0$ follows since $X_T$ and $X_S$ are independent; in the second sum on the second line, the value $1$ is an upper bound on the difference of two probabilities; and the final value is an upper bound on the number of $S$ and $T$ that intersect.

By Chebyshev's inequality, we find that 

$$\Pr\left[|X - \E[X]| \geq  \frac{\alpha - \beta}{2}\binom{k}{r}\right] \leq  \frac{ 4 \cdot {\rm Var}[X]}{(\alpha - \beta)^2\binom{k}{r}^2} \leq \frac{4 \cdot r^2}{(\alpha - \beta)^2k}.$$

By linearity of expectation, $\E[X] = \binom{k}{r} \cdot \frac{e(G)\cdot r!}{n^r}$. Since $n$ is large enough compared to $r$ and $\alpha > \beta$, we have that $e(G)\ge \alpha \binom{n}{r} \ge (\frac{\alpha + \beta}{2})\cdot \frac{n^r}{r!}$. Thus we find that $\E[X] \geq (\frac{\alpha + \beta}{2})\binom{k}{r}$. Therefore the probability that $e(G[S]) \leq \beta\binom{k}{r}$ conditioned on the $v_i$ being distinct is at most $\frac{4\cdot r^2}{(\alpha - \beta)^2k} \cdot e$. Since $k$ is large enough, we have $\sqrt{k} \ge \frac{4 \cdot e\cdot r^2}{(\alpha-\beta)^2}$. Thus, the number of $k$-subsets $S$ of $V(G)$ with $e(G[S]) \leq \beta\binom{k}{r}$ is at most $\frac{4\cdot e\cdot r^2}{(\alpha - \beta)^2 k} \cdot \binom{n}{k} \le \frac{1}{\sqrt{k}} \cdot \binom{n}{k}$, as desired.
\end{lateproof}

\section{Proof of the General Embedding Lemma}

Here we provide a ``slotting''-style proof of the General Embedding Lemma,~\cref{lem:EmbedGeneral}. First, we require the following Bipartite Finishing Lemma, which is a simplified version of Lemma~3.4 in~\cite{DP24}. 

\begin{lem}[Bipartite Finishing Lemma \cite{DP24}]\label{lem:BipFinish}
    For every integer $r \ge 1$, there exists $D_0$ such that for all integers $D \ge D_0$ the following holds:
If $G = (A, B)$ is an bipartite $r$-graph such that $d_G(a) \ge 8rD$ for each $a \in A$ and $d_G(b) \le D$ for each $b \in B$, then there exists an $A$-perfect matching of $G.$ 
\end{lem}

Lemma~\ref{lem:BipFinish} is implied by a result of Alon~\cite{A94}, which shows that $d_G(a) \ge 2erD$ suffices, and a result of Haxell~\cite{H01}, who improved this to $d_G(a) \ge 2rD.$

We are now prepared to prove~\cref{lem:EmbedGeneral}.

\begin{proof}[Proof of Lemma~\ref{lem:EmbedGeneral}]
We choose $C'$ large enough as needed throughout the proof. Let $T:= \frac{C'}{C}\cdot \Delta(J)$. For each $H\in \mc{H}$, let $\Phi_H$ be the set of embeddings of $W_H$ into $G\setminus (E(J)\setminus E(H))$. For each $\phi \in \Phi_H$, define 
$$S(\phi):= \bigg\{ S\in \binom{V(G)}{r-1}: \exists e\in E(\phi(W_H))\setminus E(H) \textrm{ with } S\subseteq e\bigg\},$$ that is, the set of $(r-1)$-subsets of $V(G)$ contained in edges of the embedding. Let $b_{\rm max} := C$ and note that $b_{\max} \ge v(W_H)-v(H)$ for all $H\in \mc{H}$ since $\mc{W}$ is $C$-bounded. Similarly let $s_{\max} := Cr$ and note that $s_{\max} \ge |S(\phi)|$ for all $H\in \mc{H}$ and $\phi \in \Phi_H$ as $\mc{W}$ is $C$-bounded.

We define an auxiliary bipartite graph $X=(A,B)$ where $A:= \mc{H}$ and $B:= E(G) \cup \left(\binom{V(G)}{r-1} \times [T]\right)$ as follows:
For each $H\in \mc{H}$, $\phi \in \Phi_H$ and $a=(a_1,\ldots, a_{|S(\phi)|})\in T^{|S(\phi)|}$, define $$e_{H,\phi,a} := \{H\} \cup E(\phi(W_H)) \cup \bigcup_{i\in [|S(\phi)|]} \{(S,a_i) : S \textrm{ is the $i$th element of } S(\phi)\}.$$ 
We let 
$$E(X) := \bigcup_{H\in \mc{H}}~\bigcup_{\phi\in \Phi_H} \bigg\{ e_{H,\phi,a} \times n^{b_{\max}-|V(W_H)\setminus V(H)|}\cdot T^{s_{\max}-|S(\phi)|}: a \in T^{|S(\phi)|} \bigg\}$$
where the times symbols denotes that the edge $e_{H,\phi,a}$ has multiplicity $n^{b_{\max}-|V(W_H)\setminus V(H)|}\cdot T^{s_{\max}-|S(\phi)|}$. Let $r':= C(1+r)+1$. Note $J$ is $r'$-bounded.

We proceed to show there exists an $A$-perfect matching $M$ of $J$ which yields an embedding $\phi = (M[H]: H\in \mc{H})$ of $\mc{W}$ into $G$ where the embedding is edge-disjoint (given $E(G)\subseteq B$ and the construction of $E(X)$) and where each $(r-1)$-subset of $V(G)$ is in at most $T\cdot C = C'\cdot \Delta(J)$ edges (since $\binom{V(G)}{r-1}\times [T]\subseteq B$ and given the construction of $E(X)$), as desired.

To that end, it suffices to use the Finishing Lemma,~\cref{lem:BipFinish}. Let $D':= \frac{1}{8r'} \cdot \gamma \cdot n^{b_{\max}}\cdot T^{s_{\max}}$. Namely, we show $d_X(a)\ge 8rD'$ for all $a\in A$ and $d_X(b)\le D'$ for all $b\in B$.

First we prove the lower bounds for degrees in $A$. By assumption, $|\Phi_H|\ge \gamma\cdot  n^{|V(W_H)\setminus V(H)|}$ for each $H\in \mc{H}$. Hence for each $H\in A=\mc{H}$, we find accounting for multiplicity and slotting that 
$$d_X(H) \ge \sum_{\phi \in \Phi_H} T^{|S(\phi)|}\cdot n^{b_{\max}-|V(W_H)\setminus V(H)|}\cdot T^{s_{\max}-|S(\phi)|} = |\Phi_H|\cdot n^{b_{\max}-|V(W_H)\setminus V(H)|}\cdot T^{s_{\max}} \ge \gamma \cdot n^{b_{\max}}\cdot T^{s_{\max}}.$$

Now we prove the upper bound for degrees in $B$. First consider the case $b=f\in E(G)$. First fix a subset $R\subseteq V(f)$, for which there are at most $2^r$ choices. Then fix an edge $e\in E(J)$ containing $R$, for which there are at most $n^{r-|R|-1}\cdot \Delta(J)$ choices. Note that such an $e$ exists because the system is edge-intersecting. Then fix $H\in \mc{H}$ containing $f$, for which there are most $C$ choices since $\mc{H}$ is $C$-refined. Then there are at most $C!\cdot n^{|V(W_H)\setminus V(H)| - |R|}$ choices of $\phi\in \Phi_H$ such that $f\in E(\phi(W_H))$. Then there are $T^{|S(\phi)|}$ slot choices and multiplicity $n^{b_{\max}-|V(W_H)\setminus V(H)|}\cdot T^{s_{\max}-|S(\phi)|}$. Altogether, then we find that
\begin{align*}
d_X(f) &\le 2^r \cdot n^{r-|R|-1}\cdot \Delta(J) \cdot C\cdot C!\cdot n^{|V(W_H)\setminus V(H)| - |R|} \cdot n^{b_{\max}-|V(W_H)\setminus V(H)|}\cdot T^{s_{\max}-|S(\phi)|}  \\
&= 2^r\cdot C\cdot C! \cdot n^{b_{\max}-1}\cdot T^{s_{\max}} \cdot \Delta(J) \le D'   
\end{align*}
as desired since $\Delta(J)$ is small enough, namely $\Delta(J)\le n/C'$ and $C' \ge 8r'\cdot 2^r\cdot C\cdot C! / \gamma$.

Finally we consider the case $b=(S,i)$ where $S\in \binom{V(G)}{r-1}$ and $i\in [T]$. Then we choose an $f\in E(G)$ with $S\subseteq V(f)$ for which there are $n$ choices; then we choose $R,e,H$ as above. Then we note there are only $T^{|S(\phi)|-1}$ slot choices. Accounting for multiplicity then, we find that 
$$d_X((S,i)) \le 2^r\cdot C\cdot C! \cdot n^{b_{\max}}\cdot T^{s_{\max}-1} \cdot \Delta(J) \le D'$$
as desired since $T=\frac{C'}{C}\cdot \Delta(J)$ and $C'\ge 8r'\cdot 2^r\cdot C^2\cdot C! / \gamma$ is large enough.
\end{proof}

\section{Proof of the Fixed Fractional Decomposition Lemma}

For completeness, here we include the proof of the Fixed Fractional Decomposition Lemma,~\cref{lem:Fixed}. 

\begin{lateproof}{lem:Fixed}
    Note that as $n$ is large enough, we have that $\varepsilon n\geq2$ and hence $\delta(G-e)\geq(\delta^*_F+\varepsilon/2)n$ for every $e\in G$. By the minimum degree of $G$, as $n$ is large enough, there exists a fractional $F$-decomposition $\Phi_0$ of $G$. Similarly, for every $e\in G$, there exists a fractional $F$-decomposition  $\Phi_{e}$ of $G-e$.

For each $e\in G$, let $\lambda_e := e(G)\Big[\varphi(e) - \big(1-\frac{1}{e(G)}\big)\Big]$. Note that since $\varphi(e) \in [1-\frac{1}{e(G)},1]$, we have that $\lambda_e\in [0,1]$. We define 
$$\Phi_e' := \lambda_e \cdot \Phi_0 + (1-\lambda_e) \cdot \Phi_e.$$ 
Note that $\Phi'_e$ is fractional $F$-packing of $G$ such that $\partial \Phi'_e(e) = \lambda_e$ and $\partial \Phi'_e(f) = 1$ for each $f\in G-e$. 
Finally we set
$$\Phi := \frac{1}{e(G)} \cdot \sum_{e\in G} \Phi'_e.$$
Observe that $\Phi$ is a fractional $F$-packing of $G$ (since it is the average of fractional $F$-packings of $G$). Furthermore for each $f\in G$, we have that
\begin{align*}
\partial \Phi(f) = \frac{1}{e(G)} \cdot \sum_{e\in G} \partial \Phi'_e(f) = \frac{1}{e(G)} \cdot \left( e(G)-1 + \lambda_f\right) = 1 - \frac{1}{e(G)} + \frac{1}{e(G)} \cdot \lambda_f = \varphi(f),
\end{align*}
as desired.
\end{lateproof}

\end{document}